\documentclass[a4paper,10pt]{article}
\usepackage[utf8]{inputenc}
\usepackage{mathtools,amsmath,amssymb,amsthm,bm,comment,tikz,subfig,pgf,float,caption}
\usepackage[a4paper,width={15cm},left=3cm,bottom=3cm, top=3cm]{geometry}
\usepackage{hyperref}

\newcommand{\R}{\mathbb{R}}	% Real numbers
\renewcommand{\phi}{\varphi}
\renewcommand{\div}{\mathrm{div}}
\newcommand{\nnu}{\bm{\nu}}	% Bold Exterior Unit
\newcommand{\<}{\langle}		% Duality
\renewcommand{\>}{\rangle}		% Duality
\newcommand{\dx}{\,\mathrm{d}x}	% dx
\newcommand{\ds}{\,\mathrm{d}S}	% ds
\newcommand{\dz}{\,\mathrm{d}z} % dz
\newcommand{\dth}{\,\mathrm{d}S(\theta)}	% d\theta
\newcommand{\vint}{v^{\textup{int}}}
\newcommand{\wint}{w^{\textup{int}}}
\newcommand{\ire}{{i,R,\epsilon}}
\newcommand{\jre}{{j,R,\epsilon}}
\newcommand{\mre}{{m,R,\epsilon}}
\newcommand{\nre}{{n,R,\epsilon}}

\DeclarePairedDelimiter\abs{\lvert}{\rvert}	% Absolute value
\DeclarePairedDelimiter\norm{\lVert}{\rVert}	% Norm

\makeatletter
\let\oldabs\abs
\def\abs{\@ifstar{\oldabs}{\oldabs*}}

\let\oldnorm\norm
\def\norm{\@ifstar{\oldnorm}{\oldnorm*}}
\makeatother

\DeclareMathOperator{\dist}{\mathrm{dist}}
\DeclareMathOperator{\const}{\mathrm{const}}

\theoremstyle{plain}
\newtheorem{theorem}{Theorem}[section]
\newtheorem{lemma}[theorem]{Lemma}
\newtheorem{corollary}[theorem]{Corollary}
\newtheorem{proposition}[theorem]{Proposition}

\theoremstyle{remark}
\newtheorem{remark}[theorem]{Remark}

\author{Veronica Felli, Roberto Ognibene}

\author{Veronica Felli\footnote{
Dipartimento di Scienza dei Materiali,
 Universit\`a di Milano--Bicocca,
Via Cozzi 55, 20125 Milano, Italy,
\texttt{veronica.felli@unimib.it}}\and Roberto Ognibene\footnote{
Dipartimento di Matematica e Applicazioni,
 Universit\`a di Milano--Bicocca,
Via Cozzi 55, 20125 Milano, Italy,
\texttt{roberto.ognibene@unimib.it}}}

\date{September 9, 2018}

\title{Sharp Convergence Rate of Eigenvalues in a Domain with a Shrinking Tube}

\begin{document}

\maketitle

\begin{abstract}
In this paper we consider  a class of singularly perturbed domains,
obtained by attaching a cylindrical tube to a fixed bounded region and letting
its section shrink to zero. 
 We use
 an Almgren-type monotonicity formula to evaluate the sharp
 convergence rate of perturbed simple eigenvalues, 
via Courant-Fischer Min-Max characterization and blow-up analysis for
scaled eigenfunctions.
\end{abstract}



\section{Introduction and Main Results}

The purpose of this work is to investigate the behaviour of the
eigenvalues of the Dirichlet-Laplacian in a class of singularly
perturbed domains: in particular we are interested in the sharp
convergence rate of the eigenvalue variation,  i.e. in the evaluation of
the leading term in its asymptotic expansion. The perturbation
consists in attaching a cylindrical tube to a fixed domain and letting
the section of the tube shrink.

Let $N\geq2$ and $\Omega\subseteq\R^N$ be open, bounded and connected. Suppose that $0\in\partial\Omega$ and that $\partial\Omega$ is flat in a neighbourhood of the origin, namely
\begin{equation}\label{eqn:hp_plane_boundary}
  \exists\,R_{\textup{max}}>1\quad\text{such that}\quad M:=\{(x_1,\dots,x_N)\in\R^N\colon x_1=0,~\abs{x}\leq R_{\textup{max}}\}\subseteq \partial\Omega.
\end{equation}
Using the following notation for the positive half-space, half-balls and half-spheres
\begin{gather*}
  \R^N_+:=\{(x_1,\dots,x_N)\in\R^N \colon x_1>0\}, \\
  B_r^+:=\{x\in \R^N_+\colon \abs{x}< r\},\qquad S_r^+:=\{x\in \R^N_+\colon \abs{x}=r\},
\end{gather*}
we can suppose, without losing generality, that
\begin{equation*}
\quad B_{R_{\textup{max}}}^+\subseteq \Omega\subseteq \R^N_+.
\end{equation*}
Let $\Sigma\subset\subset M$ be open, connected and containing the origin
$0$. For simplicity of exposition we assume that $\partial\Sigma$ is of class $C^2$; although this regularity assumption can be relaxed, see Remark \ref{rmk:lip_bound}. Moreover we assume, for sake of simplicity, that the radius of $\Sigma$ in $\R^{N-1}$ is 1, i.e.
\[
 \max_{x\in \partial\Sigma}\abs{x}=1.
\]
Finally we assume that $\Sigma$ is \textbf{starshaped} with respect to
$0$, i.e.
\begin{equation}\label{eqn:hp_sigma_starshaped}
 x\cdot\nnu\geq0 \text{ for all }x\in\partial\Sigma,
\end{equation}
where $\nnu$ denotes the exterior unit normal vector to $\partial\Sigma$. 
Let $\epsilon\in\R$, with $0<\epsilon\leq1$, and let $T_\epsilon:=(-1,0]\times\epsilon\Sigma$ be a cylindrical tube with section $\epsilon\Sigma=\{\epsilon x\colon x\in \Sigma\}$. Let us denote by
\begin{equation}\label{eqn:perturbed_dom}
  \Omega^\epsilon=\Omega\cup T_\epsilon
\end{equation}
the \emph{perturbed} domain (see Figure \ref{fig:domains}).

Let $p\in L^\infty(\R^N)$ be a weight function such that  $p\geq0$
a.e. and $p\not\equiv 0$ in $\Omega$. For any open, bounded set
$\omega\subseteq \R^N$, we consider the weighted  Dirichlet eigenvalue problem for the Laplacian on~$\omega$
\begin{equation}\label{pbm:eigen_omega}\tag{$E_\omega$}
  \left\{\begin{aligned}
    -\Delta \phi & =\lambda p \phi, & & \text{in}~\omega, \\
    \phi  &=0,  &&\text{on}~\partial\omega.
  \end{aligned}\right.
\end{equation}
By classical spectral theory
we have that, if $p\not\equiv 0$ in $\omega$, there exists a sequence of positive eigenvalues
of (\ref{pbm:eigen_omega})
\[
 0<\lambda_1(\omega)< \lambda_2(\omega)\leq \lambda_3(\omega)\leq \dots
\]
repeated according to their multiplicity.  We denote by
$(\lambda_n)_n:=(\lambda_n(\Omega))_n$ the sequence of eigenvalues of
the unperturbed problem $(E_\Omega)$, and by $(\phi_n)_n$ a
corresponding sequence of eigenfunctions such that
$\int_\Omega p\abs{\phi_n}^2 \,dx=1$ and $\int_\Omega p\phi_n \phi_m\,dx=0$ if
$n\neq m$. Similarly, we denote by
$(\lambda_n^\epsilon)_n:=(\lambda_n(\Omega^\epsilon))_n$ and
$(\phi_n^\epsilon)_n$ the sequences of eigenvalues and eigenfunctions
of the perturbed problem $(E_{\Omega^\epsilon})$, such
that $\int_\Omega p|\phi_n^\epsilon|^2 \,dx=1$ and
$\int_\Omega p\phi_n^\epsilon \phi_m^\epsilon\,dx=0$ if $n\neq m$.

Let $j\in \mathbb{N}$ be such that
\begin{equation}\label{eqn:hp_simplicity}
 \lambda_j\quad\text{is \textbf{simple}.}
\end{equation}
Assumption (\ref{eqn:hp_simplicity}) is not so restrictive: indeed, 
the simplicity of all eigenvalues is a generic property with respect
to perturbations of the domain , see \cite{Micheletti1972,Uhlenbeck1976}.

 Classical results (see for instance \cite{BUCUR_1998,DANERS_2003})
ensure the continuity
with respect to our domain perturbation, i.e. $\lambda_j^\epsilon$ is simple for $\epsilon$ small and
\begin{equation}\label{eq:7}
 \lambda_j^\epsilon\longrightarrow \lambda_j,\qquad\text{as }\epsilon\to 0.
\end{equation}
Furthermore, for every $\epsilon$   we can choose the
eigenfunction $\phi_j^\epsilon$ in such a way that 
\begin{equation}\label{eq:3}
 \phi_j^\epsilon\longrightarrow \phi_j, \qquad\text{in }H^1_0(\Omega^1),\quad\text{as }\epsilon\to 0,
\end{equation}
where the functions are trivially extended in $\Omega^1$  outside their domains. 

The main goal of this paper is to find the exact asymptotics of the difference $\lambda_j-\lambda_j^\epsilon$ as $\epsilon\to 0$.

\begin{figure}
    \subfloat[][\emph{Unperturbed Domain}]{
    \begin{tikzpicture}[scale=1.5]  
    \draw [->] (0,-1.5) -- (0,1.5) node [left] {$x_2,\dots,x_N$};
    \draw [->] (-1,0) -- (2,0) node [below] {$x_1$};
    \draw [thick] (0,0) -- (0,1) to [out=90,in=150] (1,1) to [out=330,in=30] (1.5,0.5) to [out=210,in=105] (1.75,-1) to [out=285,in=270] (0,-1) -- (0,0); 
    \draw[fill] (0,0) circle [radius=0.025];
    \node [below] at (-0.1,0) {0};
    \node [above] at (1.3,1) {$\Omega$};
    \end{tikzpicture}
   }\hspace{1cm}
    \subfloat[][\emph{Perturbed Domain}]{
    \begin{tikzpicture}[scale=1.5]
    \draw [->] (0,-1.5) -- (0,1.5) node [left] {$x_2,\dots,x_N$};
    \draw [->] (-2,0) -- (2,0) node [below] {$x_1$};
    \draw [thick] (0,0.2) -- (0,1) to [out=90,in=150] (1,1) to [out=330,in=30] (1.5,0.5) to [out=210,in=105] (1.75,-1) to [out=285,in=270] (0,-1) -- (0,-0.2);
    \draw [thick] (0,-0.2) -- (-1.5,-0.2) -- (-1.5,0.2) -- (0,0.2);
    \draw [black,fill=black,opacity=0.5] (-1.6,0) to [out=90,in=180] (-1.5,0.2) to [out=0,in=90] (-1.4,0) to [out=270,in=0] (-1.5,-0.2) to [out=180,in=270] (-1.6,0);
    \draw [black,fill=black,opacity=0.5] (-0.1,0) to [out=90,in=180] (0,0.2) to [out=0,in=90] (0.1,0) to [out=270,in=0] (0,-0.2) to [out=180,in=270] (-0.1,0);
    
    \node [above] at (-1.8,0.1) {$\epsilon\Sigma$};
    \node [above] at (1.3,1) {$\Omega^\epsilon$};
    \node [below] at (-0.75,-0.2) {$T_\epsilon $};
   \end{tikzpicture}
   }
   \caption{}
\label{fig:domains}
\end{figure}

The problem of convergence of eigenvalues and eigenfunctions of the
Dirichlet-Laplacian with respect to perturbations of the domain has
been widely studied in the past. For instance, for general
perturbations that cover the shrinking tube, in \cite{Babuska1965} the
authors investigated the stability of the spectrum with respect to
general scalar products, while \cite{Bucur2006} dealt with the
convergence of solutions of a nonlinear eigenvalue problem (see also
\cite{Dancer1988,Dancer1990}). Within an extensive literature, we
mention \cite{Hale2005}, \cite{Henry2005} and \cite{Burkenov2006} as
detailed surveys. In \cite{Davies1993,Burkenov2008} bounds for
the rate of convergence have been found; furthermore, in
\cite{Taylor2013} the framework is pretty similar to ours and the
author proved an estimate of the type
$\lambda_j-\lambda_j^\epsilon=O(\epsilon^a)$, where $a$ depends only
on the distance between $\lambda_j$ and its neighbours.  
  We mention also that asymptotic expansions of the eigenvalues of the
  Dirichlet Laplacian in domains with a thin shrinking cylindrical
  protrusion of finite length were obtained in \cite{Gadylshin}, see
  also \cite{Amirat_Chechkin_Gadylshin} for a related problem in a
  two-dimensional domain with thin shoots; 
we notice that Theorem 4.1 in  \cite{Gadylshin} provides the exact vanishing
rate of the eigenvalue variation $\lambda_j-\lambda_j^\epsilon$ only
when $\nabla\varphi_j(0)\neq0$, but it does not say what
is the leading term in the expansion when $\nabla\varphi_j(0)=0$.
  For what concerns Neumann
boundary conditions, among many others, we cite
\cite{Arrieta1995,Arrieta1995a,Jimbo1989, Jimbo1993,Nazarov1995},
which take into account singular perturbations like the shrinking
tube.

A motivation for the interest in studying the spectral behaviour of
the Laplacian on thin branching domains comes from physics: for
instance, it occurs in the theory of quantum graphs, which models the
propagation of waves in quasi one-dimensional structures, like quantum
wires, narrow waveguides, photonic crystals, blood vessels, etc. (see
e.g. \cite{Berkolaiko2013,Kuchment2005} and reference therein).
Moreover, this topic is also related with engineering problems, such
as elasticity and multi-structure problems, as well explained in
surveys \cite{Ciarlet1990,Movchan2006}.

The starting points of this work are  \cite{Gadylshin} and 
\cite{Abatangelo2014,Abatangelo2014a,Felli2013}.  On the
  one hand, the
present paper aims at  providing a criterion for selecting the leading term in
  the asymptotic expansion given in  \cite{Gadylshin}, based on the
  vanishing order of the limit eigenfunction at the junction; on the
  other hand, it improves and generalizes  
some results of
\cite{Abatangelo2014}. We note that \cite{Abatangelo2014} (as well as 
many of the aforementioned articles) deals with 
dumbbell domains in which the tubular handle is vanishing. However,
from the point of view of both the expected results and the technical
approach, our method does not require substantial adaptions to treat
also the  dumbbell case; hence for the sake of simplicity of
exposition, in the present paper we consider only perturbations of type (\ref{eqn:perturbed_dom}).

In order to state our main results, we first need to recall some known
facts. 
Let us consider the eigenvalue problem for the standard Laplacian on the $(N-1)$-dimensional unit sphere
\begin{equation}\label{eqn:eigen_sphere}
 -\Delta_{\mathbb{S}^{N-1}}\psi=\mu \psi \qquad\text{in }\mathbb{S}^{N-1}.
\end{equation}
It is well known that the eigenvalues of \eqref{eqn:eigen_sphere} are $\mu_k=k(k+N-2)$, for $k=0,1,\dots$ and that their multiplicities are (see \cite{Berger1971})
\begin{equation*}
  m_k=\binom{k+N-2}{ k}+\binom{k+N-3}{ k-1}.
\end{equation*}
If $E_k$ denotes the eigenspace of the eigenvalue $\mu_k$, then $\bigoplus_{k\geq 0} E_k=L^2(\mathbb{S}^{N-1})$. Furthermore it is known that the elements of $E_k$ are spherical harmonics, i.e. homogeneous polynomials (of $N$ variables) of degree $k$. We are interested in eigenfunctions of \eqref{eqn:eigen_sphere} that vanish on $\{x_1=0\}$, so let us call
\[
 E_k^0:=\left\{ \psi\in E_k\colon \psi(0,\theta_2,\dots,\theta_N)=0  \right\}.
\]
It is well known (see e.g. \cite[Th. 1.3]{Felli2011}) that the local
behaviour at $0\in\partial\Omega$ of eigenfunctions of $(E_\Omega)$
can be described in term of spherical harmonics vanishing on
$\{x_1=0\}$. In particular there exist  $k\in \mathbb{N}$, $k\geq 1$,
and $\Psi \in E_{k}^0$, $\Psi\neq 0$, such that
 \begin{align}
 \label{eq:1} r^{-k}\phi_j(r\theta)&\to \Psi \qquad\text{in }C^{1,\tau}(S_1^+),\quad\text{as }r\to 0^+, \\
   \label{eq:2} r^{1-k}\nabla\phi_j(r\theta)&\to \nabla\Psi \qquad\text{in }C^{0,\tau}(S_1^+,\R^N),\quad\text{as }r\to 0^+,
 \end{align}
 for all $\tau\in (0,1)$. Furthermore the asymptotic homogeneity order $k$ can be characterized as the limit of an Almgren frequency function (see \cite{Almgren1983}), i.e.
 \begin{equation*}
  \lim_{r\to 0^+}\frac{r\int_{B_r^+}\left(\abs{\nabla\phi_j}^2-\lambda_j\abs{\phi_j}^2\right)\dx}{\int_{S_r^+}\abs{\phi_j}^2\ds}=k.
 \end{equation*}
 Hereafter we will denote
 \begin{equation}\label{eqn:def_psi_k}
  \psi_{k}(r\theta):=r^{k}\Psi(\theta),\quad r\geq 0,\ \theta\in S_1^+.
\end{equation}
The exact asymptotic estimate  of the eigenvalue variation we are
going to prove  involves a
nonzero constant $m_{k}(\Sigma)$ which admits the following  variational  characterization. Let us consider the functional
\begin{equation*}
 \begin{aligned}
     J\colon &\mathcal{D}^{1,2}(\Pi)\longrightarrow \R, \\
      J(u):=&\frac{1}{2}\int_\Pi \abs{\nabla u}^2\dx-\int_\Sigma u \frac{\partial\psi_{k}}{\partial x_1}\ds,
 \end{aligned}
\end{equation*}
where
\begin{equation*}
 T_1^-:=(-\infty,0]\times\Sigma,\qquad \Pi:=T_1^-\cup \R^N_+,
\end{equation*}
and, for any open set $\omega\subseteq\R^N$,
$\mathcal{D}^{1,2}(\omega)$ denotes the completion of the space
$C_c^\infty(\omega)$ with respect to the $L^2$ norm of the gradient
(see Section \ref{section:prelim} for further details). In dimension
$2$, we will always  deal with spaces $\mathcal{D}^{1,2}(\omega)$ with
$\omega$ such that $\R^N\setminus\omega$ contains a half-line; in this
case $\mathcal{D}^{1,2}(\omega)$ can be characterized as a concrete
functional space thanks to the validity of a Hardy inequality also in
dimension 2, see Theorem \ref{thm:2dim_hardy}.

By standard minimization methods, one can prove that $J$ is bounded from below and that the infimum
\begin{equation}\label{eqn:def_m_k}
 m_{k}(\Sigma):=\inf_{u\in \mathcal{D}^{1,2}(\Pi)}J(u)
\end{equation}
is attained by some $w_{k}$. Moreover
\begin{equation}\label{eqn:m_k_integral}
m_{k}(\Sigma)=-\frac{1}{2}\int_{\Pi}\abs{\nabla w_k}^2\dx=-\frac{1}{2}\int_\Sigma \frac{\partial \psi_k}{\partial \nnu}w_k\ds <0,
\end{equation}
see \cite{Felli2013}.
With this framework in mind we are able to state our first (and main) result.
\begin{theorem}\label{thm:main_1}
 Under assumptions \eqref{eqn:hp_plane_boundary},
 \eqref{eqn:hp_sigma_starshaped} and \eqref{eqn:hp_simplicity}, let $k$ denote the vanishing order of the unperturbed eigenfunction
 $\phi_j$ as in \eqref{eq:1}--\eqref{eq:2}. Then
 \begin{equation*}
  \lim_{\epsilon\to 0}\frac{\lambda_j-\lambda_j^\epsilon}{\epsilon^{N+2k-2}}=C_{k}(\Sigma),
 \end{equation*}
 where 
 \begin{equation}\label{eqn:def_C_k}
  C_{k}(\Sigma)=-2 m_{k}(\Sigma)>0
 \end{equation}
 and $m_{k}(\Sigma)$ is defined in \eqref{eqn:def_m_k}.
\end{theorem}

We recall that, for $N\geq3$, an asymptotic expansion for the eigenvalue variation
  is constructed  using the \emph{concordance method} in \cite[Theorem
  4.1]{Gadylshin}, but explicit
  formulas are given only for the first perturbed coefficient, which turns out
  to  be a multiple of $|\nabla \varphi_j(0)|^2$; in dimension $N=2$, \cite[Theorem
  10.1]{Gadylshin} performs a more detailed asymptotic analysis
  with the computation of all the coefficients. Hence, for $N\geq3$,
  \cite{Gadylshin} finds outs what is the leading term in the
  asymptotic expansion only when $\nabla \varphi_j(0)\neq0$.
We emphasize that, differently from \cite{Gadylshin}, Theorem \ref{thm:main_1} detects the exact
  vanishing rate of $\lambda_j- \lambda_j^\epsilon$ also when $\nabla
  \varphi_j(0)=0$ and $N\geq3$; more precisely it establishes
  a direct correspondence between  the order of the infinitesimal
  $\lambda_j- \lambda_j^\epsilon$  and  the number $k$, which is the order of vanishing of
$\varphi_j$ at the junction point $0$.

The proof of Theorem \ref{thm:main_1} is based on lower and upper bounds for the difference
$\lambda_j-\lambda_j^\epsilon$ carried out using 
the
Min-Max Courant-Fischer characterization of the eigenvalues, see
Section \ref{sec:estim-diff-texorpdfs}. 
To obtain the exact asymptotics for the eigenvalue variation it is crucial to sharply  control the energy of
perturbed 
eigenfunctions  in neighbourhoods of the junction with radius of order $\epsilon$. 
The sharpness of our energy  estimates is related to the
identification of a nontrivial limit profile for blow-up of scaled
eigenfunctions, as stated in the following theorem.

\begin{theorem}\label{thm:main_2_blowup}
 Under the same assumptions of Theorem \ref{thm:main_1}, let
 $\phi_j^\epsilon$ be chosen as in \eqref{eq:3}. Then
 \begin{equation*}
\epsilon^{-k}  \phi_j^\epsilon(\epsilon x)\to \Phi(x)\qquad \text{as }\epsilon\to 0,
 \end{equation*}
in $H^1(T_1^-\cup B_R^+)$ for all $R>1$,
 where $\Phi:=w_{k}+\psi_{k}$, being  $w_{k}$  the minimizer
 for \eqref{eqn:def_m_k} and $\psi_{k}$ the homogeneous function defined  in \eqref{eqn:def_psi_k}.
\end{theorem}
As mentioned before, Theorem \ref{thm:main_1} generalizes and improves 
\cite[Th. 1.1]{Abatangelo2014}: indeed, in \cite{Abatangelo2014} the
weight $p$
was assumed to vanish in a neighbourhood of the junction $\Sigma$ 
and only the case of vanishing order $k=1$ for the unperturbed
eigenfunction $\phi_j$ was considered. Furthermore the dimension $N=2$
was not included in 
\cite{Abatangelo2014}.
As in \cite{Abatangelo2014}, a fundamental tool for the proof of the
energy estimates needed to study the local behaviour of eigenfunctions
is  an Almgren-type
monotonicity formula, which was first introduced by Almgren 
\cite{Almgren1983} and then used by Garofalo and Lin 
\cite{Garofalo1986} to study unique continuation properties for
elliptic partial differential equations.

In the particular case treated in \cite{Abatangelo2014,Felli2013}, 
precise pointwise estimates from above and
from below for the perturbed eigenfunction and its gradient were  directly obtained via comparison and maximum
principles: indeed,  if the limit eigenfunction has minimal vanishing
order at the origin and the weight vanishes 
around the junction, then such eigenfunction has a fixed sign and is
harmonic in a neighbourhood of $0$.
These estimates  were  used in \cite{Felli2013} to get rid of a
remainder term in the derivative of the Almgren quotient for the
perturbed problem, however they are not available in the more general framework of the
present paper. Nevertheless,
under the  geometric assumption \eqref{eqn:hp_sigma_starshaped} on the tube section 
we succeed in proving that the remainder term has a positive sign,
thus obtaining the monotonicity formula, see Proposition \ref{prop:monotonicity}.
We also point out that the 2-dimensional case requires the proof of an
ad hoc Hardy type inequality for functions vanishing on a fixed
half-line, see \eqref{eqn:hardy2}.

We observe that in  \cite{Abatangelo2014,Felli2013} the limit of the
blow-up family $\epsilon^{-k}\phi_j^\epsilon(\epsilon x)$ was
recognized by its frequency at infinity, which must be necessarily
equal to the minimal one, i.e. 1, in the particular case
$k=1$. In the general
case $k\geq1$, the monotonicity argument implies that the frequency
of the limit profile  is less than or equal to $k$, and this seems to
be not enough for a univocal identification. To overcome this
difficulty, we use here an argument inspired by \cite{Abatangelo2015}
and based on a local inversion result giving an energy control for the
difference between the blow-up eigenfunction and a $k$-homogeneous
profile,  see Corollary
\ref{cor:blow_up_bounded}.

The paper is organized as follows. After some preliminary results  in
Section \ref{section:prelim}, in Section \ref{sec:pohozaev} we
prove a Pohozaev-type identity, which is combined with the Poincar\'e
inequalities of Section \ref{section:poincare} to develop a
monotonicity argument 
in Section
\ref{sec:monotonicity-formula}. From the monotonicity formula
established in Corollary \ref{cor:N_bounded}, we derive some local
energy estimates which allow us to deduce sharp upper and lower bounds
for the eigenvalue variation in Section
\ref{sec:estim-diff-texorpdfs}. In Section \ref{sec:blow-up-analysis}
we perform a blow-up analysis for scaled eigenfunctions from which we
deduce first Theorem \ref{thm:main_2_blowup} and then, in Section
\ref{sec:proof-theor-refthm:m}, our main result Theorem \ref{thm:main_1}. Finally, in
the appendix 
we recall an Hardy type inequality in dimension $2$ for functions
vanishing on half-lines  and an abstract lemma on maxima of quadratic forms.

\section{Preliminaries and Notation}\label{section:prelim}

In this section we introduce some basic definitions and notation which will be useful in the rest of the paper. We start fixing some notation:
\begin{align*}
 &\Omega_r^\epsilon:=T_\epsilon\cup B_r^+, \quad \epsilon\in(0,1),\
 r\in(\epsilon,R_{\textup{max}}),\\
& \mathcal{C}_r:=\partial B_r^+\setminus S_r^+,\quad r\in(0,R_{\textup{max}}),\\
& \Pi_r:=T_1^-\cup B_r^+ ,\quad r>1.
\end{align*}
 For any measurable  set $\omega\subseteq \R^N$, we denote as
 $\abs{\omega}$ its  $N$-dimensional Lebesgue measure.

For any $R\geq 2$, we will denote as $\eta_R$ a cut-off function satisfying 
\begin{equation}\label{eqn:def_cut_off}
\begin{gathered}
 \eta_R\in C^\infty(\overline{\Pi}), \quad
 \eta_R(x)=\begin{cases}
 1, & \text{for }x\in \Pi\setminus \Pi_R, \\
 0, & \text{for }x\in \Pi_{R/2}  ,
 \end{cases}\\
 \abs{\eta_R(x)}\leq 1,\quad \abs{\nabla \eta_R(x)}\leq 4/R \qquad\text{for all }x\in \Pi.
 \end{gathered}
\end{equation}

We now recall a well known quantitative result about the first eigenvalue of the Dirichlet-Laplacian on bounded domains.

\begin{theorem}[Faber-Krahn Inequality]
 Let $\omega\subseteq \R^N$ be open and bounded and let $\lambda_1^D(\omega)$ denote the first eigenvalue of the Dirichlet-Laplacian on $\omega$. Then
 \[
  \lambda_1^D(\omega)\geq \frac{\lambda_1^D(B_1)\abs{B_1}^{2/N}}{\abs{\omega}^{2/N}},
 \]
where $B_1$ denotes the $N$-dimensional ball centered at the origin and with radius 1. 
\end{theorem}

Moreover, if we denote 
\begin{equation}\label{eqn:def_C_N}
 C_N:=\frac{1}{ \abs{B_1}^{2/N}\lambda_1^D(B_1) },
\end{equation}
 and we combine the previous Theorem with the usual Poincar\'e Inequality, we have that
\begin{equation}\label{eqn:poincare_faber_krahn}
 \int_{\omega}\abs{u}^2\dx\leq C_N \abs{\omega}^{2/N}\int_\omega \abs{\nabla u}^2\dx \quad\text{for all }u\in H^1_0(\omega).
\end{equation}

\subsection{The Space \texorpdfstring{$\mathcal{H}_R$}{The Space HR}}
For $R>1$ let us define the function space $\mathcal{H}_R$ as the completion of $C_c^\infty(\Pi_R\cup S_R^+)$ with respect to the norm induced by the scalar product
\[
 (u,v)_{\mathcal{H}_R}:= \int_{\Pi_R}\nabla u\cdot\nabla v\dx.
\]
Since $\Pi_R$ is bounded in at least 1 direction, the Poincar\'e Inequality holds. Hence $\mathcal{H}_R\hookrightarrow H^1(\Pi_R)$ continuously and we have the following characterization 
\begin{equation*}
 \mathcal{H}_R:=\left\{u\in H^1(\Pi_R)\colon u=0~\text{on}~\partial\Pi_R\setminus S_R^+ \right\}.
\end{equation*}
Moreover, when $N\geq 3$, the classical Sobolev inequality implies
that $\mathcal{H}_R\hookrightarrow L^{2^*}(\Pi_R)$ continuously, where
$2^*=\frac{2N}{N-2}$.

\subsection{Limit Profiles}

In this section we introduce some limit profiles that will appear in the blow-up analysis of scaled eigenfunctions. We recall the following result from \cite[Lemma 2.4]{Felli2013}.
\begin{proposition}\label{prop:def_Phi}
 For every $\psi\in C^2(\R^N_+)\cap C^1(\overline{\R^N_+})$  such that
 \begin{equation*}
  \left\{\begin{aligned}
          -\Delta \psi&=0, &&\text{in }\R^N_+, \\
          \psi&=0, &&\text{on }\partial \R^N_+,
         \end{aligned}\right.
 \end{equation*}
 there exists a unique $\Phi=\Phi(\psi):\Pi\to\R$ such that
  \begin{gather}
  \Phi \in \mathcal{H}_R,\qquad   \text{for all } R>1, \label{eqn:Phi_1} \\
  \left\{\begin{aligned}\label{eqn:Phi_2}
          -\Delta \Phi &=0, &&\text{in }\Pi , \\
          \Phi&=0, &&\text{on }\partial\Pi  ,        
         \end{aligned}\right. \\
         \int_\Pi  \abs{\nabla(\psi-\Phi)}^2\dx<+ \infty .\label{eqn:Phi_3}
 \end{gather}
\end{proposition}

Hereafter we will denote
\begin{equation}\label{eqn:def_Phi}
 \Phi:=\Phi(\psi_{k})
\end{equation}
where $\psi_{k}$ is the function defined in \eqref{eqn:def_psi_k}. As
observed in \cite{Felli2013} we have that 
\begin{equation}\label{eqn:Phi_components}
 \Phi=
\begin{cases}
\psi_k+w_k,&\text{in }\R^N_+,\\
w_k,&\text{in }\Pi\setminus\R^N_+,
\end{cases}
\end{equation}
where $w_k$ is the function realizing the minimum $m_k(\Sigma)$ in
\eqref{eqn:def_m_k}.
We observe that, in the particular case $N=2$, the
  function $\Phi$ corresponds to the function $X_k$ introduced in
  \cite[Sections 10-11]{Gadylshin}.
By a classical Dirichlet principle, one can easily obtain the following result.
\begin{lemma}\label{lemma:def_U_R}
  For every $R>1$ there exists a unique function $U_R\in \mathcal{H}_R$ solution to the following minimization problem
  \[
    \min_{u\in \mathcal{H}_R} \left\{\int_{\Pi_R}\abs{\nabla u}^2\dx \colon u=\psi_k~\text{on }S_R^+ \right\}.
  \]
  Moreover it weakly solves
  \[
    \left\{\begin{aligned}
             -\Delta U_R &=0, && \text{in }\Pi_R, \\
             U_R&=0 ,&& \text{on }\partial \Pi_R \setminus S_R^+, \\
             U_R &=\psi_k, &&\text{on }S_R^+.
           \end{aligned}\right.
  \]
\end{lemma}

\begin{lemma}\label{lemma:conv_U_R}
 For every $r>1$ we have
 \[
  U_R\longrightarrow \Phi \qquad\text{in }\mathcal{H}_r, \qquad\text{as }R\to +\infty.
 \]
\end{lemma}
\begin{proof}
 We can assume $R>\max\{r,2\}$. The function $U_R-\Phi$ satisfies, in a weak sense,
 \[
  \left\{\begin{aligned}
          -\Delta(U_R-\Phi)&=0, && \text{in }\Pi_R, \\
          U_R-\Phi &=0, && \text{on }\partial\Pi_R\setminus S_R^+, \\
          U_R-\Phi &= \psi_k-\Phi, &&\text{on }S_R^+,
         \end{aligned}\right.
 \]
and then it is the least energy function among those having these boundary conditions.
Let $\eta=\eta_R\in C^\infty(\overline{\Pi})$ be the cut-off function
defined in \eqref{eqn:def_cut_off}. Then
\begin{gather*}
 \int_{\Pi_r}\abs{\nabla(U_R-\Phi)}^2\dx\leq \int_{\Pi_R}\abs{\nabla(U_R-\Phi)}^2\dx\leq \int_{\Pi_R}\abs{\nabla (\eta(\psi_k-\Phi))}^2\dx \leq \\
 \leq 2\int_{\Pi_R}\abs{\nabla\eta}^2\abs{\psi_k-\Phi}^2\dx +2\int_\Pi \abs{\eta}^2\abs{\nabla (\psi_k-\Phi)}^2\dx\leq \\
 \leq \frac{32}{R^2}\int_{\Pi_R\setminus \Pi_{R/2}}\abs{\psi_k-\Phi}^2\dx +2\int_{\Pi-\Pi_{R/2}}\abs{\nabla (\psi_k-\Phi)}^2\dx \leq \\
 \leq 32\int_{\Pi\setminus\Pi_{R/2}}\frac{\abs{\psi_k-\Phi}^2}{\abs{x}^2}\dx+2\int_{\Pi-\Pi_{R/2}}\abs{\nabla (\psi_k-\Phi)}^2\dx\longrightarrow 0
\end{gather*}
thanks to \eqref{eqn:Phi_3} and Hardy's Inequality. In the case $N=2$ we use the fact that $1+\abs{x}^2\leq 2\abs{x}^2$ for $\abs{x}\geq 1$ and the 2-dimensional Hardy's Inequality \eqref{eqn:hardy2}.
\end{proof}

Using again the Dirichlet principle, we construct also the limit profile $Z_R$ as follows.

\begin{lemma}\label{lemma:def_Z_R}
  For every $R>1$ there exists a unique function $Z_R\in H^1(B_R^+)$ solution to the following minimization problem
  \[
    \min_{u\in H^1(B_R^+)} \left\{\int_{B_R^+}\abs{\nabla u}^2\dx \colon u=0~\text{on }\mathcal{C}_R,~u=\Phi ~\text{on }S_R^+ \right\}.
  \]
  Moreover it weakly solves
  \[
    \left\{\begin{aligned}
             -\Delta Z_R &=0 ,&& \text{in }B_R^+, \\
             Z_R&=0, && \text{on }\mathcal{C}_R, \\
             Z_R &=\Phi, &&\text{on }S_R^+.
           \end{aligned}\right.
  \]
\end{lemma}

\section{A Pohozaev-Type Inequality}\label{sec:pohozaev}

The purpose of this section is to prove the following inequality.

\begin{proposition}\label{prop:pohoz_inequality}
 There exists $\tilde{\epsilon},\tilde{r}>0$, with
 $0<\tilde{\epsilon}<\tilde{r}\leq R_{\textup{max}}$, such that, for
 $\epsilon\in (0,\tilde{\epsilon}]$, $r\in (\epsilon,\tilde{r}]$ and
 $i\in\{1,\dots, j\}$, we have 
 \begin{equation}\label{eqn:prop_pohoz_inequality}
  \int_{S_r^+}\abs{\nabla \phi_i^\epsilon}^2\ds -\frac{N-2}{r}\int_{\Omega_r^\epsilon}\abs{\nabla\phi_i^\epsilon}^2\dx\geq 2 \int_{S_r^+}\left( \frac{\partial \phi_i^\epsilon}{\partial\nnu}\right)^2\ds+\frac{2\lambda_i^\epsilon}{r}\int_{\Omega_r^\epsilon}p\phi_i^\epsilon\nabla\phi_i^\epsilon\cdot x\dx.
 \end{equation}
\end{proposition}
We observe that solutions to problems of type
\[
  \left\{\begin{aligned}
           -\Delta u & = f, & & \text{in }\Omega_r^\epsilon, \\
           u& =0, &&\text{on }\partial\Omega_r^\epsilon,
         \end{aligned}\right.
\]
with $f\in L^2(\Omega_r^\epsilon)$, in general do not belong to
$H^2(\Omega_r^\epsilon)$ since $\partial\Omega_r^\epsilon$ is only
Lipschitz continuous and doesn't verify a uniform exterior ball
condition (which ensures $L^2$-integrability of second order
derivatives, see \cite{ADOLFSSON_1992}). But, along a proof of a
Pohozaev Identity, one tests the equation with the function
$\nabla \phi_i^\epsilon\cdot x$, which could fail to be $H^1$ in our case. To
overcome this difficulty we implement an approximation process.

\subsection{Approximating Domains}
Let $\epsilon\in (0,1]$ and $r\in (\epsilon,R_{\textup{max}}]$: in order to remove the concave edge $\Gamma_\epsilon:=\partial (\epsilon\Sigma)$, we will approximate the domain $\Omega^\epsilon_r$ with a family of starshaped domains $\Omega^\epsilon_{r,\delta}$ (with $0\leq \delta< r-\epsilon$) such that 
\[
  \Omega^\epsilon_{r,0}=\Omega^\epsilon_{r},\quad\Omega^\epsilon_{r,\delta_1}\subset \Omega^\epsilon_{r,\delta_2}\quad\text{for all }0\leq \delta_1\leq\delta_2<r-\epsilon,
\]
and such that every $\Omega^\epsilon_{r,\delta}$ verifies the uniform exterior ball condition. In particular we will define $Q_\delta$ such that
\[
  \Omega^\epsilon_{r,\delta}=\Omega^\epsilon_r\cup Q_\delta.
\]
For $\delta>0$ small we define a ``$\delta$-enlargement'' of $\epsilon\Sigma$:
\[
 \epsilon\Sigma_\delta:=\left\{ x\in \R^{N-1}\setminus\epsilon\Sigma \colon \dist(x,\Gamma_\epsilon)<\delta\right\}.
\]
Let $h\in C^\infty((0,+\infty))\cap C([0,+\infty))$ such that $h(0)=\delta$, $h(s)=0$ for $s\in [\delta,+\infty)$, $h'(s)<0$ for all $s\in (0,\delta)$ and $h^{-1}(s)=h(s)$ for $s\in (0,\delta)$. We define $G\colon\epsilon\Sigma_\delta \subset \R^{N-1}\to \R$ as
\[
G(z):=-h(d(z)), 
\]
where $d(z)=\dist (z,\Gamma_\epsilon)$. Now, let 
\[
  Q_\delta=\left\{ x\in \R^N\colon (0,x_2,\dots,x_N)\in \epsilon\Sigma_\delta~\text{and}~ G(x_2,\dots,x_N)<x_1\leq 0\right\},
\]
\begin{figure}\hspace{0.5cm}
    \subfloat[][\emph{The function $h$}]{
    \begin{tikzpicture}[scale=1.5]
    \draw [->] (0,-1) -- (0,2) node [right] {$h(s)$};
    \draw [->] (-1,0) -- (2,0) node [below] {$s$};
    \draw [thick] (0,1) to [out=270,in=180] (1,0) -- (2,0); 
    \draw [dashed] (-0.5,-0.5) -- (1.5,1.5);
    \filldraw[black] (1,0) circle (1pt) node[anchor=north] {$\delta$};
    \filldraw[black] (0,1) circle (1pt) node[anchor=east] {$\delta$};
    \end{tikzpicture}}\hspace{2cm}
    \subfloat[][\emph{The Approximating Domain}]{\begin{tikzpicture}[scale=1.5]
    \draw [->] (0,-1.5) -- (0,1.5) node [left] {$x_2,\dots,x_N$};
    \draw [->] (-2,0) -- (2,0) node [below] {$x_1$};
    \draw [thick] (0,0.2) -- (0,1) to [out=90,in=150] (1,1) to [out=330,in=30] (1.5,0.5) to [out=210,in=105] (1.75,-1) to [out=285,in=270] (0,-1) -- (0,-0.2);
    \draw [thick] (0,-0.2) -- (-1.5,-0.2) -- (-1.5,0.2) -- (0,0.2);
    \draw [black,fill=black,opacity=0.5] (-1.6,0) to [out=90,in=180] (-1.5,0.2) to [out=0,in=90] (-1.4,0) to [out=270,in=0] (-1.5,-0.2) to [out=180,in=270] (-1.6,0);
    	\draw [thick] (-0.4,0.2) to [out=0,in=270] (0,0.6);
    \draw [thick] (-0.4,-0.2) to [out=0,in=90] (0,-0.6);
    
    \draw [thick,red] (0,0.8) to [out=0,in=90] (0.8,0) to [out=270,in=0] (0,-0.8);
    \draw [thick,dotted] (0,0) -- (0.64,0.46);
    
    \node [below] at (0.3,0.2) {\scriptsize $r$};
    \node [below] at (0.8,-0.6) {$\Omega^\epsilon_{r,\delta}$};
    
    \node [above] at (-0.3,0.3) {$Q_\delta$};
    \node [above] at (0.75,0.5) {\small $S_r^+$};
   \end{tikzpicture}
   }\caption{}
   \label{fig:approx_domains}
\end{figure}
see Figure \ref{fig:approx_domains}.

For what concerns the regularity we observe that, since $\Gamma_\epsilon$ is of class $C^2$, then also $d$ and the graph of $G$ are of class $C^2$ (see \cite{Krantz1981}). Moreover it is easy to verify that the approximating domain $\Omega_{r,\delta}^\epsilon$ satisfies the uniform exterior ball condition.

\begin{remark}\label{rmk:starshapeness}
 We point out that $\Omega_{r,\delta}^\epsilon$ is starshaped. Indeed, first, if $x\in \mathcal{C}_r\setminus \left( \epsilon\Sigma_\delta \cup \epsilon\Sigma\right)$, trivially $x\cdot \nnu(x)=0$. If $x\in \{x_1=-\epsilon\}\cap \partial\Omega_{r,\delta}^\epsilon$, then $x\cdot\nnu(x)=-x_1=\epsilon >0$. Third, if $x\in (-\epsilon,-\delta)\times\Gamma_\epsilon$, then $\nnu(x)=\nnu_0(x')$, where $x'=(x_2,\dots,x_N)$ and $\nnu_0(x')$ is the exterior unit normal of $\Gamma_\epsilon$; thus $x\cdot \nnu(x)=x'\cdot \nnu_0(x')\geq 0$ by \eqref{eqn:hp_sigma_starshaped}. Finally, let $x\in \{(G(x'),x')\colon x'\in \epsilon\Sigma_\delta\}$: in this case we have that
\[
 \nnu(x)=\frac{(-1,-h'(d(x'))\nabla d(x'))}{\norm{(-1,-h'(d(x'))\nabla d(x'))}},
\]
and so
\[
 x\cdot \nnu(x)=\frac{-G(x')-h'(d(x'))\nabla d(x')\cdot x'}{\norm{(-1,-h'(d(x'))\nabla d(x'))}}\geq 0,
\]
where we used the properties of the functions $G$ and $h$ and the fact
that, since $0\in \Sigma$ and $\Sigma$ is starshaped, $\frac{\partial
  d}{\partial x_i}x_i\geq 0$ for any $i=2,\dots,N$, see \cite[Proof of
Lemma 14.16]{GILBARG_2015}.
\end{remark}

\subsection{Approximating Problems}\label{sec:approx_prob}
For $\alpha\in (0,1)$, let us fix $\tilde{r},\tilde{\epsilon}>0$, with $0<\tilde{\epsilon}<\tilde{r}\leq R_{\textup{max}}$, such that
\begin{equation}\label{eqn:hp_D_delta}
  \abs{\Omega_{r,\delta}^\epsilon}\leq \left( \frac{1-\alpha}{C_N\lambda_j \norm{p}_\infty} \right)^{\frac{N}{2}} \quad\text{for all }\epsilon\in(0,\tilde{\epsilon}),~r\in (\epsilon,\tilde{r}),~\delta\in(0,r-\epsilon),
\end{equation}
where $C_N$ has been defined in \eqref{eqn:def_C_N}.

For fixed $i\in\{1,\dots,j\}$, $\epsilon\in (0,\tilde{\epsilon}]$, $r\in (\epsilon, \tilde{r}]$ and for all $\delta \in (0,r-\epsilon)$, let us consider the problem

\begin{equation}\label{pbm:eigen_D_delta}
  \left\{
  \begin{aligned}
    -\Delta u &=\lambda_i^\epsilon p u, & & \text{in}~D_\delta, \\
    u & =0, & &\text{on}~\partial D_\delta\setminus S_r^+, \\
    u & = \phi_i^\epsilon, & & \text{on}~S^+_r,
  \end{aligned}
  \right.
\end{equation}
where, for simplicity of notation, in this section we call
$D_\delta:=\Omega_{r,\delta}^\epsilon$; we also denote
$\delta_0:=r-\epsilon$.
\begin{theorem}
  There exists a unique $u_\delta\in H^1(D_\delta)$ solution to problem \eqref{pbm:eigen_D_delta}. Moreover the family $\{u_\delta\}_{\delta\in(0,\delta_0)}$ is bounded in $H^1(D_{\delta_0})$ with respect to $\delta$.
\end{theorem}
\begin{proof}
  We observe that, if we extend $\phi_i^\epsilon$ to zero in $D_{\delta_0}\setminus \Omega_r^\epsilon$ and we let $v=u-\phi_i^\epsilon$, then problem \eqref{pbm:eigen_D_delta} is equivalent to
    \begin{equation}\label{pbm:eigen_D_delta_v}
    \left\{
    \begin{aligned}
      -\Delta v &=\lambda_i^\epsilon p v +F, & & \text{in}~D_\delta, \\
      v & =0, & &\text{on}~\partial D_\delta,
    \end{aligned}
    \right.
  \end{equation}
  where $F:=\lambda_i^\epsilon p \phi_i^\epsilon+\Delta\phi_i^\epsilon\in H^{-1}(D_\delta)$. Existence and uniqueness of a solution $v_\delta\in H^1_0(D_\delta)$ to \eqref{pbm:eigen_D_delta_v} easily comes from Lax-Milgram Theorem. Indeed, the bilinear form
  \[
   a(v,w):=\int_{D_\delta}\left(\nabla v\cdot\nabla w-\lambda_i^\epsilon p vw\right)\dx,\quad v,w\in H^1_0(D_\delta)
  \]
  is coercive, since, by \eqref{eqn:poincare_faber_krahn} and \eqref{eqn:hp_D_delta}, we have
  \begin{equation}\label{eqn:coercivity_lax}
  \begin{split}
   a(v,v)&=\int_{D_\delta}\left(\abs{\nabla v}^2-\lambda_i^\epsilon p\abs{v}^2\right)\dx \\ &\geq \left( 1-C_N\lambda_j\norm{p}_\infty \abs{D_\delta}^{\frac{2}{N}} \right)\int_{D_\delta}\abs{\nabla v}^2\dx\geq \alpha \int_{D_\delta}\abs{\nabla v}^2\dx.
   \end{split}
  \end{equation}
 From Lax-Milgram Theorem we also know that
  \[
    \norm{\nabla v_\delta}_{L^2(D_{\delta_0})}=\norm{\nabla v_\delta}_{L^2(D_\delta)}\leq \frac{\norm{F}_{H^{-1}(D_\delta)}}{\alpha},
  \]
  where $v_\delta$ has been trivially extended in
  $D_{\delta_0}\setminus D_\delta$. One can easily prove that
  $\norm{F}_{H^{-1}(D_\delta)}=O(1)$ as $\delta\to 0$. Then
  $u_\delta:=v_\delta+\phi_i^\epsilon$ is the unique solution to
  \eqref{pbm:eigen_D_delta} and
  $\{ u_\delta\}_{\delta\in(0,\delta_0)}$ is bounded in
  $H^1(D_{\delta_0})$.
\end{proof}

\begin{theorem}\label{thm:strong_conv_u_delta}
If $u_\delta\in H^1(D_\delta)$ is the unique solution to \eqref{pbm:eigen_D_delta}, then
  \[
    u_\delta\longrightarrow \phi_i^\epsilon\qquad\text{in }H^1(D_{\delta_0}), \quad\text{as }\delta\to 0.
  \]
\end{theorem}
\begin{proof}
Since $\{ u_\delta\}_{\delta\in(0,\delta_0)}$ is bounded in $H^1(D_{\delta_0})$, then there exists $U\in H^1(D_{\delta_0})$ such that, up to a subsequence,
\[
 u_\delta \rightharpoonup U\qquad\text{weakly in }H^1(D_{\delta_0}),\quad \text{as }\delta \to 0.
\]
Actually $U\in H^1(\Omega_r^\epsilon)$ and moreover it weakly solves
\[
 \left\{\begin{aligned}
  -\Delta U &=\lambda_i^\epsilon p U, &&\text{in }\Omega_r^\epsilon, \\
  U &=0, && \text{on }\partial\Omega_r^\epsilon\setminus S_r^+, \\
  U & =\phi_i^\epsilon, &&\text{on }S_r^+.
 \end{aligned}\right.
\]
From the coercivity obtained in \eqref{eqn:coercivity_lax} we deduce that $U=\phi_i^\epsilon$.

In order to prove strong convergence in $H^1(D_{\delta_0})$, we notice that
\begin{equation}\label{eqn:aux_str_conv}
  \int_{D_{\delta_0}}\abs{\nabla u_\delta}^2\dx=\lambda_i^\epsilon\int_{D_{\delta_0}}p u_\delta^2\dx+\int_{S_r^+}\phi_i^\epsilon \frac{\partial u_\delta}{\partial \nnu}\ds.
\end{equation}
From the compactness of the embedding $H^1(D_{\delta_0})\hookrightarrow L^2(D_{\delta_0})$ we have that $u_\delta\to \phi_i^\epsilon$ in $L^2(D_{\delta_0})$ and so $\int_{D_{\delta_0}}p \abs{u_\delta}^2\dx \to \int_{D_{\delta_0}}p \abs{\phi_i^\epsilon}^2\dx$. Moreover, from the equation \eqref{pbm:eigen_D_delta}, we have that $\nabla u_\delta\rightharpoonup \nabla \phi_i^\epsilon$ weakly in $H(\div,D_{\delta_0})$ as $\delta\to 0$. Hence classical trace theorems for vector functions yield 
\[
  \int_{S_r^+}\frac{\partial u_\delta}{\partial\nnu}\phi_i^\epsilon\ds \to \int_{S_r^+}\frac{\partial \phi_i^\epsilon}{\partial\nnu}\phi_i^\epsilon\ds,\quad\text{as }\delta\to 0.
\]
Therefore, from \eqref{eqn:aux_str_conv} and from the equation satisfied by $\phi_i^\epsilon$, we conclude that, along a subsequence,
\[
\lim_{\delta\to 0}\int_{D_{\delta_0}}\abs{\nabla u_\delta}^2\dx
=\lambda_i^\epsilon\int_{D_{\delta_0}}p\abs{\phi_i^\epsilon}^2\dx+\int_{S_r^+}\frac{\partial
  \phi_i^\epsilon}{\partial\nnu}\phi_i^\epsilon\ds=\int_{\Omega_r^\epsilon}\abs{\nabla\phi_i^\epsilon}^2\dx=\int_{D_{\delta_0}}\abs{\nabla
  \phi_i^\epsilon}^2\dx.
\]
Thanks to \emph{Urysohn's Subsequence Principle} the proof is thereby complete.
\end{proof}

\begin{theorem}\label{thm:conv_grad_boundary}
Let $u_\delta\in H^1(D_\delta)$ be the unique solution to \eqref{pbm:eigen_D_delta}. Then
  \[
    \nabla u_\delta \longrightarrow \nabla \phi_i^\epsilon\qquad\text{in }L^2(S^+_r),\quad\text{as }\delta\to 0.
  \]
\end{theorem}
\begin{proof}
By classical elliptic regularity theory, it is easy to prove that an odd reflection of $u_\delta$ through the hyperplane $\{x_1=0\}$ in a neighbourhood of $\{x\colon \abs{x}=r\}$ converges to $\phi_i^\epsilon$ in $H^2$, as $\delta\to 0$. Hence the conclusion follows by trace embeddings.
\end{proof}

\begin{theorem}\label{thm:pohoz_identity}
  Let $u_\delta$ be the unique solution to
  \eqref{pbm:eigen_D_delta}. Then the following identity holds
  \begin{multline*}
    \int_{S^+_r}\abs{\nabla u_\delta}^2\ds-\frac{N-2}{r}\int_{D_\delta}\abs{\nabla u_\delta}^2\dx\\
   =\frac{1}{r}\int_{\partial D_\delta\setminus S_r^+}\left(\frac{\partial u_\delta}{\partial \nnu}\right)^2x\cdot \nnu \ds +2\int_{S^+_r}\left(\frac{\partial u_\delta}{\partial \nnu}\right)^2\ds+\frac{2\lambda_i^\epsilon}{r}\int_{D_\delta}p\,u_\delta\,\nabla u_\delta\cdot x\dx.
 \end{multline*}
\end{theorem}
\begin{proof}
  We first observe that, by classical regularity theory, $u_\delta\in H^2(D_\delta)$ since $D_\delta$ verifies an exterior ball condition. Let us now test equation \eqref{pbm:eigen_D_delta} with the function $\nabla u_\delta \cdot x\in H^1(D_\delta)$. Integrating by parts and using the following identity
   \[
    \nabla u_\delta\cdot\nabla(\nabla u_\delta\cdot x)=\frac{1}{2}\div\left({\abs{\nabla u_\delta}^2 x}\right)-\frac{N-2}{2}\abs{\nabla u_\delta}^2
  \]
  we obtain the conclusion.
\end{proof}

\begin{proof}[Proof of Proposition \ref{prop:pohoz_inequality}]
  Let $u_\delta$ be the unique solution to \eqref{pbm:eigen_D_delta}. From Theorem \ref{thm:pohoz_identity} and Remark \ref{rmk:starshapeness} we know that
  \begin{gather*}
    \int_{S^+_r}\abs{\nabla u_\delta}^2\ds-\frac{N-2}{r}\int_{D_\delta} \abs{\nabla u_\delta}^2\dx   
    \geq 2\int_{S^+_r}\left(\frac{\partial u_\delta}{\partial \nnu}\right)^2\ds+ \frac{2\lambda_i^\epsilon}{r}\int_{D_\delta}p\,u_\delta\,\nabla u_\delta\cdot x\dx.
  \end{gather*}
  Now, thanks to Theorems \ref{thm:strong_conv_u_delta} and
  \ref{thm:conv_grad_boundary}, we can pass to the limit as $\delta\to
  0$ in the  above inequality, thus obtaining \eqref{eqn:prop_pohoz_inequality}.
\end{proof}

\begin{remark}\label{rmk:lip_bound}
 We observe that the assumption of $C^2$-regularity for $\partial \Sigma$ can be relaxed; indeed, if $\Sigma$ is less regular (e.g. if $\partial\Sigma$ is Lipschitz continuous), we can approximate $\epsilon\Sigma$ with a class of $C^2$- regular domains $(\epsilon\Sigma)_\beta$ and start the procedure of Section \ref{sec:pohozaev} from the domain $(\epsilon\Sigma)_\beta$. 
\end{remark}

\section{Poincar\'e-Type Inequalities}\label{section:poincare}

In this section we consider the following spaces for
$\epsilon\in (0,1]$ and $r>\epsilon$:
\[
V_\epsilon(B_r^+):=\left\{u\in H^1(B_r^+)\colon
  u=0~\text{on}~\mathcal{C}_r\setminus \epsilon\Sigma \right\}, \quad
V_0(B_r^+):=\{ u\in H^1(B_r^+)\colon u=0~\text{on}~\mathcal{C}_r\}.
\]
 We point out that, for $0\leq\epsilon_1\leq\epsilon_2<r$,
\begin{equation}\label{eqn:inclusion_V}
  H^1_0(B_r^+)\subseteq V_0(B_r^+)\subseteq V_{\epsilon_1}(B_r^+)\subseteq V_{\epsilon_2}(B_r^+)\subseteq H^1(B_r^+).
\end{equation}

\begin{lemma}[Poincar\'e-Type Inequality]
  Let $r>0$. Then, for every $u\in H^1(B_r^+)$,
  the following inequality holds
  \begin{equation}\label{eqn:poincare_type}
   \frac{N-1}{r^2}\int_{B_r^+}\abs{u}^2\dx\leq\int_{B_r^+}\abs{\nabla u}^2\dx+\frac{1}{r}\int_{S_r^+}\abs{u}^2\ds.
  \end{equation}

\end{lemma}
\begin{proof}
Integrating the equality 
$\div (u^2x)=2u\nabla u\cdot x+Nu^2$ over $B_r^+$
 and recalling the elementary inequality
$  0\leq (u+\nabla u\cdot x)^2=\abs{u}^2+\abs{\nabla u\cdot
  x}^2+2u\nabla u\cdot x$, we obtain that 
 \[
  \int_{\partial B_r^+}\abs{u}^2 x\cdot \nnu \ds=\int_{B_r^+}\left(2u\nabla u\cdot x+N\abs{u}^2\right)\dx\geq -\int_{B_r^+}\left(\abs{u}^2+\abs{\nabla u\cdot x}^2\right)\dx+N\int_{B_r^+}\abs{u}^2\dx.
 \]
 Since $x\cdot\nnu=0$ on $\mathcal{C}_r$ and $\abs{x}\leq r$, then
 \[
r  \int_{S_r^+}\abs{u}^2\ds\geq -r^2\int_{B_r^+}\abs{\nabla u}^2\dx+(N-1)\int_{B_r^+}\abs{u}^2\dx.
 \]
 Reorganizing the terms and dividing by $r^2$ yields the thesis. 
\end{proof}

\begin{lemma}
 For $0\leq  \sigma<1$ the infimum
 \[
  m_\sigma=\inf_{\substack{u\in V_\sigma (B_1^+) \\ u\neq 0 }} \frac{\int_{B_1^+}\abs{\nabla u}^2\dx}{\int_{S_1^+}\abs{u}^2\ds}
 \]
 is achieved. Moreover $m_\sigma>0$, the map $\sigma\mapsto m_\sigma$ is non-increasing in $[0,1)$ and continuous in 0 and $m_0=1$.
\end{lemma}
\begin{proof}
 For $u\in V_\sigma(B_1^+)$, let us denote
 \[
  F(u):=\frac{\int_{B_1^+}\abs{\nabla u}^2\dx}{\int_{S_1^+}\abs{u}^2\ds}.
 \]
 Let $\{ u_n\}_n\subseteq V_\sigma (B_1^+)$ be a minimizing sequence such that $\int_{S_1^+}\abs{u_n}^2\ds=1$. From \eqref{eqn:poincare_type} it follows that $\{ u_n\}_n$ is bounded in $H^1(B_1^+)$ and so there exists $\tilde{u}\in H^1(B_1^+)$ such that, up to a subsequence, $u_n\rightharpoonup \tilde{u}$ in $H^1(B_1^+)$. Taking into account the compact embedding $H^1(B_1^+)\hookrightarrow L^2(\partial B_1^+)$, we have that $\int_{S_1^+}\abs{\tilde{u}}^2\ds=1$ and then $\tilde{u}\neq 0$. Moreover $\tilde{u}=0$ on $\mathcal{C}_r\setminus \sigma \Sigma$, since $\{u_n\}$ do; in particular $\tilde{u}\in V_\sigma (B_1^+)$. By weak lower semicontinuity, we have that
 \[
 \int_{B_1^+}\abs{\nabla \tilde{u}}^2\dx\leq \liminf_{n\to +\infty}\int_{B_1^+}\abs{\nabla u_n}^2\dx=m_\sigma.
 \]
 Then $m_\sigma=F(\tilde{u})$, i.e. $\tilde{u}$ attains the infimum $m_\sigma$. Trivially $m_\sigma>0$,
 due to the null boundary conditions on
 $\mathcal{C}_r\setminus\sigma\Sigma$. The monotonicity of the map
 $\sigma\mapsto m_\sigma$ follows from \eqref{eqn:inclusion_V}.
 
 Now we have to prove continuity. Let $\sigma_n\to0^+$. For every $n$ there exists $\tilde{u}_n \in V_{\sigma_n}(B_1^+)$ such that
 \[
  \int_{S_1^+}\abs{\tilde{u}_n}^2\ds=1,\qquad\int_{B_1^+}\abs{\nabla \tilde{u}_n}^2\dx=m_{\sigma_n}.
 \]
 Then, since $m_{\sigma_n}\leq m_0$ for all $n$, we have that $\{\tilde{u}_n\}_n$ is bounded in $H^1(B_1^+)$. Thus there exists $u_0\in H^1(B_1^+)$ sucht that, up to a subsequence, $\tilde{u}_n\rightharpoonup u_0$ weakly in $H^1(B_1^+)$.
 So, first
 \[
  \int_{S_1^+}\abs{u_0}^2\ds=1\qquad \text{and}\qquad u_0=0\quad\text{on}\quad \mathcal{C}_r.
 \]
 Furthermore
 \[
  m_0 \leq\int_{B_1^+}\abs{\nabla u_0}^2\dx\leq \liminf m_{\sigma_n}\leq\limsup m_{\sigma_n}\leq m_0.
 \]
 Then, along the subsequence, $m_0=\lim_{n\to +\infty} m_{\sigma_n}$. Thanks to \emph{Urysohn's Subsequence Principle} we may conclude that $m_0=\lim_{\sigma\to 0^+}m_\sigma$.
 
 Finally we prove that $m_0=1$. Since the function $u_0$ achieving $m_0$ is harmonic in $B_1^+$, thanks to classical monotonicity arguments (we refer to \cite{Felli2011} for further details) we can say that the function
 \[
  r\longmapsto M(r):=\frac{r\int_{B_r^+}\abs{\nabla u_0}^2\dx}{\int_{S_r^+}\abs{u_0}^2\ds}
 \]
 is non-decreasing and that there exists $k\in \mathbb{N}$, $k\geq 1$, such that
$\lim_{r\to 0^+}M(r)=k$.
 Hence $M(r)\geq 1$ for every $r\geq 0$. Then
 \[
  m_0=\frac{\int_{B_1^+}\abs{\nabla u_0}^2\dx}{\int_{S_1^+}\abs{u_0}^2\ds}=M(1)\geq 1.
 \]
 Furthermore the function $v(x)=x_1$ belongs to $V_0(B_1^+)$ and $F(v)=1$; hence $m_0=1$.
\end{proof}

\begin{corollary}\label{cor:ineq_rho}
Let $\epsilon\in (0,1]$ and $r>\epsilon$. Then
 \[
  \frac{m_{\epsilon/r}}{r}\int_{S_r^+}\abs{u}^2\ds\leq \int_{B_r^+}\abs{\nabla u}^2\dx\qquad\text{for all }u\in V_\epsilon(B_r^+).
 \]
 Moreover, for every $\rho\in(0,1)$ there exists $\mu_\rho>1$ such that, if $\epsilon<\frac{r}{\mu_\rho}$, then
 \begin{equation}\label{eqn:m_cont}
  \frac{1-\rho}{r}\int_{S_r^+}\abs{u}^2\ds\leq \int_{B_r^+}\abs{\nabla u}^2\dx\qquad\text{for all }u\in V_\epsilon(B_r^+).
  \end{equation}
\end{corollary}
\begin{proof}
 If we let $\sigma=\epsilon/r$ in the previous Lemma, we have that
 \[
  m_{\epsilon/r}\int_{S_1^+}\abs{u}^2\ds\leq\int_{B_1^+}\abs{\nabla u}^2\dx\qquad\text{for all }u\in V_{\epsilon/r}(B_1^+).
 \]
 The first inequality follows by the change of variables $y=rx$, while the second one trivially comes from the continuity of $m_\sigma$ in $0$.
\end{proof}

From \eqref{eqn:poincare_type} and Corollary \ref{cor:ineq_rho} one
can easily prove the following corollary.

\begin{corollary}
   For every $\rho\in(0,1)$ there exists $\mu_\rho>1$ such that, for 
every $r>0$ and 
$\epsilon<\frac{r}{\mu_\rho}$,
  \begin{equation}\label{eqn:poincare_type_1}
    \frac{N-1}{r^2}\int_{B_r^+}\abs{v}^2\dx\leq\left(1+\frac{1}{1-\rho}\right)
    \int_{B_r^+}\abs{\nabla v}^2\dx\quad\text{for all } v\in V_\epsilon(B_r^+).
  \end{equation}
\end{corollary}

\section{Monotonicity Formula}\label{sec:monotonicity-formula}

For any $0<\epsilon<r\leq R_{\textup{max}}$, $\lambda>0$, $\phi\in H^1(\Omega_r^\epsilon)$, let us introduce the functions
\[
  E(\phi,r,\lambda,\epsilon):=\frac{1}{r^{N-2}}\int_{\Omega^\epsilon_{r}}\left(\abs{\nabla \phi}^2-\lambda p\abs{\phi}^2\right)\dx
\]
and
\[
  H(\phi,r):=\frac{1}{r^{N-1}}\int_{S^+_r}\abs{\phi}^2\ds.
\]
Moreover we define the Almgren type frequency function
\[
  \mathcal{N}(\phi,r,\lambda,\epsilon):=\frac{E(\phi,r,\lambda,\epsilon)}{H(\phi,r)}.
\]

\begin{lemma}[Integration on the Tube]\label{lemma:int_tube}
   There exists a constant $\kappa=\kappa(N,\Sigma)>0$, depending only on $N$ and $\abs{\Sigma}$, such that, for every $\epsilon\in (0,1]$ and for every $u\in H^1(T_\epsilon)$ such that $u=0$ on $\partial T_\epsilon\setminus\epsilon\Sigma$,
  \[
   \int_{T_\epsilon}\abs{u}^2\dx\leq\kappa\epsilon^{2(N-1)/N} \int_{T_\epsilon}\abs{\nabla u}^2\dx.
  \]
\end{lemma}
\begin{proof}
  Let $\tilde{T}_\epsilon=T_\epsilon \cup \varsigma(T_\epsilon)$, where $\varsigma$ is the reflection through the hyperplane $\{x_1=0\}$, and let $\tilde{u}$ be the even extension of $u$ on $\tilde{T}_\epsilon$. Since $\tilde{u}\in H^1_0(\tilde{T}_\epsilon)$, thanks to \eqref{eqn:poincare_faber_krahn} we have that
  \[
   \int_{T_\epsilon}\abs{u}^2\dx=\frac{1}{2}\int_{\tilde{T}_\epsilon}\abs{\tilde{u}}^2\dx\leq \frac{C_N}{2}|\tilde{T}_\epsilon|^{2/N}\int_{\tilde{T}_\epsilon}\abs{\nabla \tilde{u}}^2\dx =C_N 2^{2/N}\abs{\Sigma}^{2/N}\epsilon^{2(N-1)/N}\int_{T_\epsilon}\abs{\nabla u}^2\dx.
  \]
Hence we can conclude the proof letting $\kappa =C_N 2^{2/N}\abs{\Sigma}^{2/N}$.
\end{proof}

\begin{lemma}\label{lemma:H_well_posed1}
 There exists $\epsilon_1\in(0,1],~R_1>0$, with $0<\epsilon_1< R_1\leq R_{\textup{max}}$, such that
 \[
  H(\phi_i^\epsilon,r)>0\qquad\text{for all }\epsilon\in(0,\epsilon_1],\quad\text{for all }r\in(\epsilon,R_1],\quad\text{for all }i=1,\dots,j.
 \]
\end{lemma}
\begin{proof}
  Suppose by contradiction that for every $n$ there exists
  $\epsilon_n\in (0,1],~ r_n\in (\epsilon_n,R_{\textup{max}}]$ and
  $i_n\in\{1,\dots,j\}$ such that $r_n\to 0$ and
  $H(\phi_{i_n}^{\epsilon_n},r_n)= 0$. Let us denote
  $\nu_n:=\lambda_{i_n}^{\epsilon_n}$,
  $\xi_n:=\phi_{i_n}^{\epsilon_n}$ and
  $\Omega_n:=\Omega_{r_n}^{\epsilon_n}$. From this it follows that
  $\xi_n=0$ on $S_{r_n}^+$ and that
 \[
  \int_{\Omega_n} \abs{\nabla \xi_n}^2\dx=\nu_n\int_{\Omega_n}p\abs{\xi_n}^2\dx\leq \lambda_j\norm{p}_\infty\int_{\Omega_n}\abs{\xi_n}^2\dx.
 \]
  Using Lemma \ref{lemma:int_tube} when integrating on the tube we obtain
  \[
   \int_{T_{\epsilon_n}}\abs{\xi_n}^2\dx\leq \kappa\epsilon_n^{2(N-1)/N}\int_{\Omega_n} \abs{\nabla \xi_n}^{2}\dx.
  \]
  Moreover \eqref{eqn:poincare_type} says that
  \[
   \int_{B_{r_n}^+}\abs{\xi_n}^2\dx\leq \frac{r_n^2}{N-1}\int_{\Omega_n} \abs{\nabla \xi_n}^{2}\dx.
  \]
  Then we have that
  \[
   \int_{\Omega_n}\abs{\nabla \xi_n}^2\dx\leq\lambda_j\norm{p}_\infty \left(\kappa\epsilon_n^{2(N-1)/N}+\frac{r_n^2}{N-1}\right)\int_{\Omega_n}\abs{\nabla \xi_n}^{2}\dx.
  \]
  Thus $\xi_n\equiv 0$ in $\Omega_n$, provided $n$ is sufficiently
  large. Thanks to classical unique continuation properties for elliptic
  equations it follows that $\xi_n=0$ in $\Omega^{\epsilon_n}$, which is
  a contradiction.
\end{proof}

\begin{lemma}\label{lemma:H_well_posed2}
 Let 
 \[
  R_2=\min\left\{\left(\frac{N-1}{\lambda_j\norm{p}_\infty}\right)^{1/2},R_{\textup{max}}\right\}.
 \]
 For every $r\in(0,R_2]$ there exist $c_r>0$ and $\epsilon_r\in (0,1]$, with $\epsilon_r<r$, such that
 \[
  H(\phi_i^\epsilon,r)\geq c_r \qquad\text{for all }\epsilon\in(0,\epsilon_r),\quad\text{for all }i=1,\dots,j.
 \]
\end{lemma}
\begin{proof}
  We will prove the lemma for a fixed $i\in\{1,\dots,j\}$ and take as
  $c_r$ the minimum among the constants found for each $i$.  Suppose
  by contradiction that for a certain ${r}\in(0,R_2]$ and for every $n$
  (large enough) there exists $\epsilon_n\in(0,1/n)$ such that
  \begin{equation}\label{eqn:lemma_H_const}
      H(\phi_i^{\epsilon_n},{r})<\frac{1}{n}.
  \end{equation}
  We first note that, since $\epsilon_n\to 0$, then
  $\lambda_i^{\epsilon_n}\to\lambda_i$ (see
  \cite{DANERS_2003}). Moreover
  \[
  \int_{\Omega^1}\abs{\nabla\phi_i^{\epsilon_n}}\dx =
  \int_{\Omega^{\epsilon_n}}\abs{\nabla\phi_i^{\epsilon_n}}\dx=\lambda_i^{\epsilon_n}\int_{\Omega^{\epsilon_n}}p\abs{\phi_i^{\epsilon_n}}^2=\lambda_i^{\epsilon_n}\leq
  \lambda_j.
  \]
  Hence there exists $v\in H^1_0(\Omega^1)$ such that
  $\phi_i^{\epsilon_n}\rightharpoonup v$ weakly in $H^1_0(\Omega^1)$
  along a subsequence. Note that actually $v\in H^1_0(\Omega)$ and
  $v\not\equiv 0$ in $\Omega$, since $\int_\Omega
  p\abs{v}^2\dx=1$.
  Moreover $\phi_i^{\epsilon_n}\to v$ strongly in $L^2(\Omega^1)$ and in
  $L^2(S_r^+)$, so that \eqref{eqn:lemma_H_const} implies that $v=0$ on
  $S_r^+$. This tells us that $v$ weakly solves
  \[\left\{\begin{aligned}
     -\Delta v&=\lambda_i pv, && \text{in }B_r^+, \\
     v&=0, &&\text{on } \partial B_r^+.
    \end{aligned}\right.  
  \]
  Testing the above equation with $v$ we obtain that
  $\int_{B_{{r}}^+}\abs{\nabla v}^2\dx=\lambda_i\int_{B_{{r}}^+}
  p\abs{v}^2\dx$
  and then, thanks to \eqref{eqn:poincare_type} and the fact that 
  $\lambda_i\leq\lambda_j$ and ${r}\leq R_2$, 
  \begin{align*}
    0=\int_{B_{{r}}^+}\left(\abs{\nabla
    v}^2-\lambda_ip\abs{v}^2\right)\dx&\geq 
\left( \frac{N-1}{r^2}-\lambda_j\norm{p}_\infty
                                        \right)\int_{B_{{r}}^+}\abs{v}^2\dx\\
  &\geq\left( \frac{N-1}{R_2^2}-\lambda_j\norm{p}_\infty \right)\int_{B_{{r}}^+}\abs{v}^2\dx.
  \end{align*}
  Due to the initial choice of $R_2$ we have that this last factor is
  positive; then $v=0$ in $B_{{r}}^+$.  This, together with classical
  unique continuation principles, implies that $v=0$ in $\Omega$,
  which is a contradiction.
\end{proof}

Let $\tilde{\epsilon}$ and $\tilde{r}$ be the constants found in Proposition \ref{prop:pohoz_inequality}.

\begin{proposition}\label{prop:E'_ineq}
Let $i\in \{1,\dots,j\}$, $\epsilon\in (0,\tilde{\epsilon}]$ and $r\in (\epsilon,\tilde{r}]$. Then
  \begin{multline}\label{eqn:th_E'}
    \frac{\mathrm{d} E}{\mathrm{d} r}(\phi_i^\epsilon,r,\lambda_i^\epsilon,\epsilon)\geq \frac{1}{r^{N-2}}\left[2\int_{S^+_r}\left(\frac{\partial \phi_i^\epsilon}{\partial \nnu}\right)^2\ds+\right. \\
    \left.+ \frac{2\lambda_i^\epsilon}{r}\int_{\Omega_r^\epsilon}p\,\phi_i^\epsilon\,\nabla \phi_i^\epsilon\cdot x\dx-\lambda_i^\epsilon\int_{S^+_r}p\abs{\phi_i^\epsilon}^2\ds+\frac{N-2}{r}\lambda_i^\epsilon \int_{\Omega^\epsilon_r}p\abs{\phi_i^\epsilon}^2\dx\right]
  \end{multline}
  and
  \begin{equation}\label{eqn:th_H'}
   \frac{\mathrm{d} H}{\mathrm{d} r}(\phi_i^\epsilon,r)=\frac{2}{r^{N-1}}\int_{S_r^+}\phi_i^\epsilon\frac{\partial \phi_i^\epsilon}{\partial \nnu}\ds =\frac{2}{r}E(\phi_i^\epsilon,r,\lambda_i^\epsilon,\epsilon).
  \end{equation}

\end{proposition}
\begin{proof}
  We compute the derivative 
  \[
    \frac{\mathrm{d} E}{\mathrm{d} r}=\frac{2-N}{r^{N-1}}\int_{\Omega^\epsilon_{r}}\left(\abs{\nabla \phi_i^\epsilon}^2-\lambda_i^\epsilon p \abs{\phi_i^\epsilon}^2\right)\dx+\frac{1}{r^{N-2}}\int_{S^+_r}\left(\abs{\nabla \phi_i^\epsilon}^2-\lambda_i^\epsilon p \abs{\phi_i^\epsilon}^2\right)\ds.
  \]
  Then, thanks to \eqref{eqn:prop_pohoz_inequality}, we obtain \eqref{eqn:th_E'}. The proof of \eqref{eqn:th_H'} follows from direct computations, the equation satisfied by $\phi_i^\epsilon$ and integration by parts.
\end{proof}

\begin{lemma}\label{lemma:estim_1}
  Let $\rho\in(0,1/2]$, $\mu_\rho$ be as in Corollary \ref{cor:ineq_rho}, $\epsilon\in(0,1]$ and $r\in(\epsilon,R_{\textup{max}}]$. If $\epsilon\mu_\rho<r$, then 
 \[
  \int_{\Omega_r^\epsilon} \abs{u}^2\dx\leq K_{\epsilon,r}^1\int_{\Omega_r^\epsilon}\abs{\nabla u}^2\dx
 \]
 for any $u\in H^1(\Omega_r^\epsilon)$ such that $u=0$ on
 $\partial\Omega_r^\epsilon\setminus S_r^+$, where
\[
  K_{\epsilon,r}^1=\kappa\epsilon^{2(N-1)/N} +\frac{3r^2}{N-1}
 \]
  and $\kappa$ is as in Lemma \ref{lemma:int_tube}. 
\end{lemma}
\begin{proof}
  Thanks to Lemma \ref{lemma:int_tube} we have an estimate about the integral on the tube, i.e.
  \begin{equation*}
   \int_{T_\epsilon}\abs{u}^2\dx\leq\kappa \epsilon^{2(N-1)/N} \int_{T_\epsilon}\abs{\nabla u}^2\dx\leq \kappa \epsilon^{2(N-1)/N} \int_{\Omega_r^\epsilon}\abs{\nabla u}^2\dx.
  \end{equation*}
  On the other hand, by \eqref{eqn:poincare_type_1} we have that
  \[
   \int_{B_r^+}\abs{u}^2\dx\leq \frac{r^2}{N-1}\left(1+\frac{1}{1-\rho}\right)\int_{B_r^+}\abs{\nabla u }^2\dx\leq \frac{3r^2}{N-1}\int_{\Omega_r^\epsilon} \abs{\nabla u}^{2}\dx.
  \]
  The conclusion follows by adding the two parts.
\end{proof}

\begin{lemma}\label{lemma:estim_2}
  Let $\rho\in(0,1/2)$, $\mu_\rho$ be as in Corollary \ref{cor:ineq_rho}, $\epsilon\in(0,1]$ and $r\in(\epsilon,R_{\textup{max}}]$. If $\epsilon\mu_\rho<r$, then
  \[
   \int_{\Omega_r^\epsilon}\abs{ u \nabla u\cdot x}\dx\leq K_{\epsilon,r}^2 \int_{\Omega_r^\epsilon}\abs{\nabla u}^2\dx
  \]
  for any $u\in H^1(\Omega_r^\epsilon)$ such that $u=0$ on $\partial\Omega_r^\epsilon\setminus S_r^+$, where
  \[
   K_{\epsilon,r}^2=\sqrt{2\kappa}\epsilon^{(N-1)/N}+\sqrt{\frac{3}{N-1}} r^2
  \]
  and $\kappa$ is as in Lemma \ref{lemma:int_tube}.
\end{lemma}
\begin{proof}
  First we consider the integral over $T_\epsilon$: thanks to
  Cauchy-Schwarz Inequality and  Lemma \ref{lemma:int_tube} we know
  that
 \begin{equation*}
  \int_{T_\epsilon} \abs{ u \nabla u\cdot x}\dx\dx\leq \sqrt{2\kappa} \epsilon^{(N-1)/N}\int_{\Omega_r^\epsilon}\abs{\nabla u}^2\dx.
 \end{equation*}
From the Cauchy-Schwarz Inequality and \eqref{eqn:poincare_type_1} it
follows that
 \[
  \int_{B_r^+}\abs{ u \nabla u\cdot x}\dx   \leq  \sqrt{\frac{3}{N-1}}r^2\int_{\Omega_r^\epsilon}\abs{\nabla u}^2\dx.
 \]
 Adding the two parts we conclude the proof.
\end{proof}

\begin{corollary}\label{cor:E_coerciv}
  Let $\rho\in(0,1/2)$, $\mu_\rho$ be as in Corollary \ref{cor:ineq_rho}, $\epsilon\in(0,1]$, $r\in(\epsilon,R_{\textup{max}}]$. If $\epsilon\mu_\rho<r$ then
  \begin{equation}\label{eqn:E_coerciv_1}
   \int_{\Omega_r^\epsilon}\left(\abs{\nabla u}^2-\lambda_i^\epsilon p\abs{u}^2\right)\dx\geq \left(1-\lambda_i^\epsilon\norm{p}_\infty K_{\epsilon,r}^1 \right) \int_{\Omega_r^\epsilon}\abs{\nabla u}^2\dx
  \end{equation}
  for any $u\in H^1(\Omega_r^\epsilon)$ such that $u=0$ on
  $\partial\Omega_r^\epsilon\setminus S_r^+$ and for all
  $i\in \{1,\dots,j\}$. Furthermore there exists
  $r_0\leq R_{\textup{max}}$ such that, for every $r,\epsilon$
  satisfying $\epsilon\mu_\rho<r\leq r_0$, we have
  \[
    \int_{\Omega_r^\epsilon}\abs{\nabla u}^2\dx\leq 2\int_{\Omega_r^\epsilon}\left(\abs{\nabla u}^2-\lambda_i^\epsilon p\abs{u}^2\right)\dx 
  \]
    for any $u\in H^1(\Omega_r^\epsilon)$ such that $u=0$ on $\partial\Omega_r^\epsilon\setminus S_r^+$ and for all $i\in \{1,\dots,j\}$.
\end{corollary}
\begin{proof}
  The first statement \eqref{eqn:E_coerciv_1} easily comes from Lemma \ref{lemma:estim_1}. Besides, if we choose $r_0\leq R_{\textup{max}}$ such that
  \[
   K_{\epsilon,r}^1\leq K_{r_0,r_0}^1\leq \frac{1}{2\lambda_j\norm{p}_\infty},
  \]
  from \eqref{eqn:E_coerciv_1}, we can conclude the proof.
\end{proof}

\begin{lemma}\label{lemma:H_increment}
  Let $\rho\in(0,1/2)$ and $\mu_\rho$ be as in Corollary
  \ref{cor:ineq_rho}. Let $R_1$ and $\epsilon_1$ be as in Lemma
  \ref{lemma:H_well_posed1}. Then there exists $\tau>0$ depending only
  on $N$, $\lambda_j$, $\|p\|_\infty$ and $|\Sigma|$ such that, for
  every $\epsilon\in (0,\epsilon_1]$, $r_1,r_2$, with
  $0<\mu_\rho\epsilon< r_1\leq r_2\leq \min\{1,R_1\}$, we have that
 \[
 \frac{H(\phi_i^\epsilon,r_2)}{H(\phi_i^\epsilon,r_1)}\geq
 \exp\left(-\tau
   R_1^{2(N-1)/N}\right)\left(\frac{r_2}{r_1}\right)^{2(1-\rho)}
 \quad\text{for all $i\in\{1,\dots,j\}$}.
\]
\end{lemma}
\begin{proof}
 With the notation
 $E(r)=E(\phi_i^\epsilon,r,\lambda_i^\epsilon,\epsilon)$,
 $H(r)=H(\phi_i^\epsilon,r)$ and $\frac{\mathrm{d} H}{\mathrm{d}
   r}(\phi_i^\epsilon,r)=H'(r)$, 
 we have that, from Proposition \ref{prop:E'_ineq} and Corollary \ref{cor:E_coerciv}
 \[
  H'(r)=\frac{2}{r}E(r)=\frac{2}{r^{N-1}}\int_{\Omega_r^\epsilon}\left( \abs{\nabla \phi_i^\epsilon}^2-\lambda_i^\epsilon p\abs{\phi_i^\epsilon}^2 \right)\dx\geq \frac{2}{r^{N-1}}(1-\lambda_i^\epsilon\norm{p}_\infty K_{\epsilon,r}^1)\int_{\Omega_r^\epsilon}\abs{\nabla \phi_i^\epsilon}^2\dx
 \]
for all $\epsilon\mu_\rho<r\leq \min\{1,R_1\}$.
 Hence, since $\lambda_i^\epsilon\leq\lambda_j$, thanks to \eqref{eqn:m_cont}
 \[
  H'(r)\geq \frac{2}{r^{N-1}}(1-\lambda_j\norm{p}_\infty K_{\epsilon,r}^1)\frac{1-\rho}{r}\int_{S_r^+}\abs{\phi_i^\epsilon}^2\ds=\frac{2(1-\rho)}{r}(1-\lambda_j\norm{p}_\infty K_{\epsilon,r}^1)H(r).
 \]
 So we have
 \[
  \frac{H'(r)}{H(r)}\geq\frac{2(1-\rho)}{r}\left[1-\tau_1\epsilon^{2(N-1)/N}-\tau_2 r^2\right]
 \]
 where $\tau_1=\lambda_j\norm{p}_\infty\kappa$ and
 $\tau_2=\lambda_j\norm{p}_\infty\frac{3}{N-1}$. Since $\epsilon<r$
 and $r\leq 1$, if $\tau_0=\tau_1+\tau_2$, then
 \[
  \left( \log H(r)
  \right)'=\frac{H'(r)}{H(r)}\geq\frac{2(1-\rho)}{r}\left[1-\tau_0
    r^{2(N-1)/N}\right]
\geq \frac{2(1-\rho)}{r}-2\tau_0 r^{1-\frac{2}{N}}.
 \]
Integrating from $r_1$ to $r_2$ and letting $\tau:=\tau_0 N/(N-1)$, we
obtain 
 \[
  \log \frac{H(\phi_i^\epsilon,r_2)}{H(\phi_i^\epsilon,r_1)}\geq 2(1-\rho)\log\frac{r_2}{r_1}-\tau(r_2^{2(N-1)/N}-r_1^{2(N-2)/N})\geq 2(1-\rho)\log\frac{r_2}{r_1}-\tau R_1^{2(N-1)/N}.
 \]
 Taking the exponentials yields the thesis.
\end{proof}

Hereafter let $R_0:=\min\{1,R_1,R_2,r_0\}$ where $R_1,R_2,r_0$ are defined in Lemma \ref{lemma:H_well_posed1}, Lemma \ref{lemma:H_well_posed2} and Corollary \ref{cor:E_coerciv} respectively. Moreover let $\epsilon_0=\min \{1,\tilde{\epsilon},\epsilon_1\}$ where $\tilde{\epsilon},\epsilon_1$ are defined in Proposition \ref{prop:pohoz_inequality} and Lemma \ref{lemma:H_well_posed1} respectively.

\begin{proposition}\label{prop:monotonicity}
   Let $\rho\in(0,1/2)$ and $\mu_\rho$ be as in Corollary \ref{cor:ineq_rho}. Then, for every $r\in (0,R_0]$, $\epsilon\in (0,\epsilon_0]$ such that $0<\epsilon\mu_\rho<r\leq R_0$
  \[
    \frac{\mathrm{d} \mathcal{N}}{\mathrm{d} r}(\phi_i^\epsilon,r,\lambda_i^\epsilon,\epsilon)\geq - f(r)\mathcal{N}(\phi_i^\epsilon,r,\lambda_i^\epsilon,\epsilon) \quad\text{for all } i\in \{1,\dots,j\},
  \]
  where
  \[
    f(r)=c_1 r +c_2 r^{(N-2)/N}+c_3 r^{-1/N}
  \]
  and $c_n$'s are positive constants depending only on $\rho$,
  $\norm{p}_\infty$, $\lambda_j$, the dimension $N$
 and the geometry of the problem (in particular on $\Omega$ and on $\abs{\Sigma}_{N-1}$).
\end{proposition}
\begin{proof}
     With the usual notation
    \[
    \frac{\mathrm{d} \mathcal{N}}{\mathrm{d} r}(\phi_i^\epsilon,r,\lambda_i^\epsilon,\epsilon)=:\mathcal{N}'(r),\quad \frac{\mathrm{d} E}{\mathrm{d} r}(\phi_i^\epsilon,r,\lambda_i^\epsilon,\epsilon):=E'(r),\quad \frac{\mathrm{d} H}{\mathrm{d} r}(\phi_i^\epsilon,r):=H'(r),
    \]
    from Proposition \ref{prop:E'_ineq} we have that
    \begin{align*}
      \mathcal{N}'(r)&\geq
                       \frac{1}{H^2}\frac{2}{r^{2N-3}}\left\{\left[\bigg(\int_{S^+_r}
\left(\frac{\partial \phi_i^\epsilon}{\partial \nnu}\right)^2\ds\bigg)
\bigg(\int_{S^+_r}\abs{\phi_i^\epsilon}^2\ds\bigg)-\left(\int_{S_r^+}
 \phi_i^\epsilon\frac{\partial \phi_i^\epsilon}{\partial \nnu}\ds \right)^2\right] \right.\\
                     &\quad +\left.\left[
                       \frac{\lambda_i^\epsilon}{r}\int_{\Omega_r^\epsilon}p\,\phi_i^\epsilon\,\nabla
                       \phi_i^\epsilon\cdot
                       x\dx-\frac{\lambda_i^\epsilon}{2}
\int_{S^+_r}p\abs{\phi_i^\epsilon}^2\ds+\frac{N-2}{2r}\lambda_i^\epsilon
                       \int_{\Omega^\epsilon_r}p\abs{\phi_i^\epsilon}^2\dx\right]
\int_{S^+_r}\abs{\phi_i^\epsilon}^2\ds\right\}.
    \end{align*}
    By Cauchy-Schwarz Inequality we have that
    \begin{equation*}
      \mathcal{N}'(r)\geq \frac{2\lambda_i^\epsilon}{\int_{S^+_r}\abs{\phi_i^\epsilon}^2}\left[ \int_{\Omega_r^\epsilon}p\,\phi_i^\epsilon\,\nabla \phi_i^\epsilon\cdot x\dx+\frac{N-2}{2} \int_{\Omega^\epsilon_r}p\abs{\phi_i^\epsilon}^2\dx-\frac{r}{2}\int_{S^+_r}p\abs{\phi_i^\epsilon}^2\ds\right].
    \end{equation*}
 Thanks to Lemmas \ref{lemma:estim_1}, \ref{lemma:estim_2}, Corollary
 \ref{cor:ineq_rho} and
 Corollary \ref{cor:E_coerciv} we can say that
 \begin{align*}
   \mathcal{N}'(r)&\geq -\frac{2\lambda_i^\epsilon\norm{p}_\infty}{\int_{S^+_r}\abs{\phi_i^\epsilon}^2}\left[K_{\epsilon,r}^2+\frac{(N-2)}{2}K_{\epsilon,r}^1+\frac{r^2}{2(1-\rho)}\right]\int_{\Omega_r^\epsilon}\abs{\nabla\phi_i^\epsilon}^2\dx  \\
                  &  \geq -\frac{4\lambda_i^\epsilon\norm{p}_\infty}{r^{N-1}H(r)}r^{N-2}E(r)\left[K_{\epsilon,r}^2+\frac{(N-2)}{2} K_{\epsilon,r}^1+r^2 \right].
 \end{align*}
 Taking into account that $K_{\epsilon,r}^n<K_{r,r}^n$, $n=1,2,$ we have
   \[
    \mathcal{N}'(r)\geq -\left(c_1 r+c_2 r^{(N-2)/N}+c_3 r^{-1/N}\right)\mathcal{N}(r) =-f(r)\mathcal{N}(r)
  \]
  by some constants $c_1,c_2,c_3>0$ independent of $r$ and $\epsilon$.
\end{proof}

\begin{corollary}\label{cor:N_bounded}
  Let $\rho\in(0,\frac12)$ and $\mu_\rho$ be as in Corollary
  \ref{cor:ineq_rho}. Then, for every $\mu>\mu_\rho$, $r\in (0,R_0]$,
  and $\epsilon\in (0,\epsilon_0]$ such that $\epsilon\mu\leq r\leq R_0$, we have that
  \[
   \mathcal{N}(\phi_i^\epsilon,r,\lambda_i^\epsilon,\epsilon)\leq e^{\int_r^{R_0}f(t)\,dt}\mathcal{N}(\phi_i^\epsilon,R_0,\lambda_i^\epsilon,\epsilon).
  \]
\end{corollary}
\begin{proof}
Form Proposition \ref{prop:monotonicity} it follows that
$\big(e^{-\int_r^{R_0}f(t)\,dt}\mathcal{N}(r)\big)'\geq 0$,
  which, by integration over $(r,R_0)$, yields the conclusion.
\end{proof}

\subsection{Energy Estimates}

\begin{proposition}\label{prop:precise_energy_estim_eps}
  Let $\rho\in(0,1/2)$. Then there exists $K_\rho>0$ such that, for
  every $R\geq K_\rho$ and for every $i\in \{1,\dots, j \}$, we have
 \begin{gather}
   \int_{\Omega_{R\epsilon}^\epsilon}\abs{\nabla \phi_i^\epsilon}^2\dx=O(\epsilon^{N-2}H(\phi_i^\epsilon,K_\rho \epsilon))\quad \text{as }\epsilon\to 0^+ ,\label{eqn:precise_energy_estim_th1}\\
   \int_{\Omega_{R\epsilon}^\epsilon}\abs{ \phi_i^\epsilon}^2\dx=O(\epsilon^{N-\frac{2}{N}}H(\phi_i^\epsilon,K_\rho \epsilon)) \quad \text{as }\epsilon\to 0^+ ,\label{eqn:precise_energy_estim_th2}\\
   \int_{S_{R\epsilon}^+}\abs{\phi_i^\epsilon}^2\ds=O(\epsilon^{N-1}H(\phi_i^\epsilon,K_\rho
   \epsilon)) \quad \text{as }\epsilon\to 0^+
   .\label{eqn:precise_energy_estim_th3}
 \end{gather}
\end{proposition}
\begin{proof}

  For $\rho\in(0,1/2)$ let us consider $\mu_\rho$ as in Corollary
  \ref{cor:ineq_rho},
  $\epsilon_0=\min \{1,\tilde{\epsilon},\epsilon_1\}$,
  $\epsilon_{R_0}$ as in Lemma \ref{lemma:H_well_posed2} and let
  $K_\rho>\max\{\mu_\rho,R_0/\epsilon_0,R_0/\epsilon_{R_0}\}$. From
  Corollary \ref{cor:N_bounded} we deduce that, if $R\geq K_\rho$ and
  $R\epsilon< R_0$
 \begin{equation}\label{eqn:energy_estim_1}
  \mathcal{N}(\phi_i^\epsilon,R\epsilon,\lambda_i^\epsilon,\epsilon)\leq e^{\int_{R\epsilon}^{R_0}f(t)\,dt}\mathcal{N}(\phi_i^\epsilon,R_0,\lambda_i^\epsilon,\epsilon).
 \end{equation}
 Now let us analyze the frequency function $\mathcal{N}$ at radius $R_0$:
 \[
  E(\phi_i^\epsilon,R_0,\lambda_i^\epsilon,\epsilon)=\frac{1}{R_0^{N-2}}\int_{\Omega_{R_0}^\epsilon}\left( \abs{\nabla\phi_i^\epsilon}^2-\lambda_i^\epsilon\abs{\phi_i^\epsilon}^2 \right)\dx\leq\frac{1}{R_0^{N-2}}\int_{\Omega^\epsilon}\abs{\nabla\phi_i^\epsilon}^2\dx\leq\frac{\lambda_j}{R_0^{N-2}}.
 \]
 Moreover, thanks to Lemma \ref{lemma:H_well_posed2}
 \[
  H(\phi_i^\epsilon,R_0)\geq c_{R_0}.
 \]
 Thus we have that
 \begin{equation}\label{eqn:energy_estim_0}
   \mathcal{N}(\phi_i^\epsilon,R_0,\lambda_i^\epsilon,\epsilon)
   \leq\frac{\lambda_j}{c_{R_0}R_0^{N-2}}.
 \end{equation}
 Then, from \eqref{eqn:energy_estim_1}
 \begin{equation}\label{eqn:energy_estim_2}
   \int_{\Omega_{R\epsilon}^\epsilon}\left(
     \abs{\nabla\phi_i^\epsilon}^2-\lambda_i^\epsilon\abs{\phi_i^\epsilon}^2
   \right)\dx\leq
   \const H(\phi_i^{\epsilon},R\epsilon)(R\epsilon)^{N-2}.
 \end{equation}
 From the second statement of Corollary \ref{cor:E_coerciv} we have that
 \begin{equation}\label{eqn:energy_estim_3}
  \int_{\Omega_{R\epsilon}^\epsilon}\abs{\nabla\phi_i^\epsilon}^2\dx 
\leq 2\const  H(\phi_i^{\epsilon},R\epsilon)(R\epsilon)^{N-2}.
 \end{equation}
 Now let $K_\rho \epsilon\leq r\leq R_0$. Then, from Proposition \ref{prop:E'_ineq}
 \begin{equation*}
  \frac{H'(\phi_i^\epsilon,r)}{H(\phi_i^\epsilon,r)}=\frac{2}{r}\mathcal{N}(\phi_i^\epsilon,r,\lambda_i^\epsilon,\epsilon)
 \end{equation*}
 and from Corollary \ref{cor:N_bounded} and \eqref{eqn:energy_estim_0}
 \begin{equation}\label{eqn:energy_estim_6}
  \frac{H'(\phi_i^\epsilon,r)}{H(\phi_i^\epsilon,r)}\leq \frac{C}{r}.
 \end{equation}
 Now, integrating the previous  inequality from $K_\rho\epsilon$ to $R\epsilon$, we obtain
 \begin{equation*}
  \log \frac{H(\phi_i^\epsilon,R\epsilon)}{H(\phi_i^\epsilon,K_\rho\epsilon)}\leq C\log \frac{R\epsilon}{K_\rho\epsilon},
 \end{equation*}
 hence
 $ H(\phi_i^\epsilon,R\epsilon)\leq \const_{\rho,R}
 H(\phi_i^\epsilon,K_\rho\epsilon)$,
 i.e.
 $H(\phi_i^\epsilon,R\epsilon)=O(H(\phi_i^\epsilon,K_\rho\epsilon))$
 as $\epsilon\to 0$.  Then \eqref{eqn:precise_energy_estim_th1}
 follows from \eqref{eqn:energy_estim_3}, whereas
 \eqref{eqn:precise_energy_estim_th3} is a direct consequence of the
 previous estimate and definition of $H$.  Finally, thanks to Lemma
 \ref{lemma:estim_1} and \eqref{eqn:precise_energy_estim_th1}, we have
 \[
  \int_{\Omega_{R\epsilon}^\epsilon}\abs{\phi_i^\epsilon}^2\dx\leq (c_1\epsilon^{2(N-1)/N}+c_2(R\epsilon)^2)\int_{\Omega_{R\epsilon}^\epsilon}\abs{\nabla\phi_i^\epsilon}\dx=O(\epsilon^{N-\frac{2}{N}}H(\phi_i^\epsilon,K_\rho\epsilon)),
 \]
 as $\epsilon\to 0$, thus proving \eqref{eqn:precise_energy_estim_th2}.
\end{proof}

\begin{proposition}\label{prop:energy_estim_eps}
  Let $\rho\in(0,1/2)$ and $K_\rho$ be as in Proposition
  \ref{prop:precise_energy_estim_eps}. Then there exists $C_\rho>0$
  such that, for every $R\geq K_\rho$, for every
  $\epsilon\in (0,\epsilon_0]$ such that $R\epsilon\leq R_0$, and for
  every $i\in\{1,\dots,j\}$ we have
  \begin{align*}
   & \int_{\Omega_{R\epsilon}^\epsilon}\abs{\nabla \phi_i^\epsilon}^2\dx\leq C_\rho(R\epsilon)^{N-2\rho}, \\
   & \int_{\Omega_{R\epsilon}^\epsilon}\abs{\phi_i^\epsilon}^2\dx\leq C_\rho(R\epsilon)^{N+2-2\rho-2/N}, \\
   & \int_{S_{R\epsilon}^+}\abs{\phi_i^\epsilon}^2\dx\leq C_\rho(R\epsilon)^{N+1-2\rho}. 
  \end{align*}
\end{proposition}
\begin{proof}
 
 From Lemma \ref{lemma:H_increment} we know that
 \begin{equation}\label{eqn:energy_estim_4}
  H(\phi_i^\epsilon,R\epsilon)\leq \exp\left(\tau R_1^{2(N-1)/N}\right)\left( \frac{R\epsilon}{R_0} \right)^{2(1-\rho)}H(\phi_i^\epsilon,R_0)
 \end{equation}
 and, from \eqref{eqn:m_cont}, we have
 \begin{equation}\label{eqn:energy_estim_5}
  H(\phi_i^\epsilon,R_0)=\frac{1}{R_0^{N-1}}\int_{S_{R_0}^+}\abs{\phi_i^\epsilon}^2\ds\leq\frac{1}{R_0^{N-2}(1-\rho)}\int_{\Omega_{R_0}^\epsilon}\abs{\nabla\phi_i^\epsilon}^2\dx\leq\frac{\lambda_j}{R_0^{N-2}(1-\rho)}.
 \end{equation}
 Combining \eqref{eqn:energy_estim_4} and \eqref{eqn:energy_estim_5}
 with  Proposition \ref{prop:precise_energy_estim_eps} (in particular
 estimate \eqref{eqn:energy_estim_3} in the proof) we can deduce all the claims.
\end{proof}

As a consequence of Proposition \ref{prop:energy_estim_eps} we can say that
 \begin{equation}\label{eqn:H_bigO}
   H(\phi_i^\epsilon,K_\rho\epsilon)=O(\epsilon^{2-2\rho})\quad\text{as }\epsilon\to 0.
 \end{equation}

As a byproduct of the proof of Proposition
\ref{prop:precise_energy_estim_eps} we obtain the following result.
\begin{corollary}\label{cor:H_control_below}
 Let $\rho\in (0,1/2)$ and $K_\rho$ be as in Proposition \ref{prop:precise_energy_estim_eps}. Then there exist $\bar{C},q>0$ such that, if $\epsilon\in (0,\epsilon_0]$ and $K_\rho\epsilon<R_0$,
 \begin{equation}\label{eqn:H_control_below}
  H(\phi_j^\epsilon,K_\rho \epsilon)\geq \bar{C}\epsilon^q.
 \end{equation}
\end{corollary}
\begin{proof}
  If we integrate \eqref{eqn:energy_estim_6} over
  $(K_\rho\epsilon,R_0)$ and take the exponentials, we obtain
 \[
 \frac{H(\phi_j^\epsilon,R_0)}{H(\phi_j^\epsilon,K_\rho\epsilon)}\leq
 \left( \frac{R_0}{K_\rho \epsilon} \right)^q,
 \]
 denoting by $q$ the constant $C$ in \eqref{eqn:energy_estim_6}.
Then Lemma \ref{lemma:H_well_posed2} implies that
 \[
 H(\phi_j^\epsilon,K_\rho\epsilon)\geq c_{R_0} \left(
   \frac{K_\rho\epsilon}{R_0} \right)^q.
 \]
 Hence the claim is proved with $\bar{C}:=c_{R_0}(K_\rho/R_0)^q$.
\end{proof}

\section{Estimates on the Difference \texorpdfstring{$\lambda_j-\lambda_j^\epsilon$}{Estimates on the Difference of Eigenvalues}}\label{sec:estim-diff-texorpdfs}

\subsection{Upper Bound}
For any $i\in\{0,\dots,j\}$, $R>1$ and $\epsilon\in (0,1]$, with
$R\epsilon\leq R_{\textup{max}}$, let us consider the following
minimization problem

\begin{equation*}
 \min \left\{\int_{B_{R\epsilon}^+}\abs{\nabla u}^2\dx\colon u\in H^1(B_{R\epsilon}^+),~u=0~\text{on}~\mathcal{C}_{R\epsilon},~u=\phi_i^\epsilon~\text{on}~S_{R\epsilon}^+\right\}.
\end{equation*}
One can prove that this problem has a unique solution
$v^{\textup{int}}_{i,R,\epsilon}$, which weakly solves
\begin{equation*}
 \left\{\begin{aligned}
         -\Delta \vint_\ire & =0, && \text{in }B_{R\epsilon}^+, \\
         \vint_\ire & = 0, && \text{on }\mathcal{C}_{R\epsilon} ,\\
         \vint_\ire & = \phi_i^\epsilon, && \text{on }S_{R\epsilon}^+.
        \end{aligned}\right.
\end{equation*}

\begin{proposition}\label{prop:precise_energy_estim_vint}
Let $\rho\in(0,1/2)$ and $K_\rho$ be as in Proposition \ref{prop:precise_energy_estim_eps}. Then
\begin{gather}
 \int_{B_{R\epsilon}^+}\abs{\nabla\vint_\ire}^2\dx=O(\epsilon^{N-2}H(\phi_i^\epsilon,K_\rho\epsilon)) \quad \text{as }\epsilon\to 0^+,\label{eqn:precise_energy_estim_vint_th1}\\
 \int_{B_{R\epsilon}^+}\abs{\vint_\ire}^2\dx=O(\epsilon^{N}H(\phi_i^\epsilon,K_\rho\epsilon))\quad \text{as }\epsilon\to 0^+,\label{eqn:precise_energy_estim_vint_th2} \\
 \int_{S_{R\epsilon}^+}\abs{\vint_\ire}^2\dx=O(\epsilon^{N-1}H(\phi_i^\epsilon,K_\rho\epsilon))\quad \text{as }\epsilon\to 0^+.\label{eqn:precise_energy_estim_vint_th3}
\end{gather}
for all $R\geq 2$ and for any $i=1,\dots,j$. Moreover there exists
$\hat{C}_\rho$ such that, if $R\geq \max\{2,K_\rho\}$ and
$\epsilon<R_0/R$,
\begin{gather}
  \label{eq:4} \int_{B_{R\epsilon}^+}\abs{\nabla\vint_\ire}^2\dx\leq
  \hat{C}_\rho(R\epsilon)^{N-2\rho}, \\
\label{eq:5} \int_{B_{R\epsilon}^+}\abs{\vint_\ire}^2\dx\leq
 \hat{C}_\rho(R\epsilon)^{N+2-2\rho}, \\
\label{eq:6} \int_{S_{R\epsilon}^+}\abs{\vint_\ire}^2\dx\leq
\hat{C}_\rho(R\epsilon)^{N+1-2\rho}.
\end{gather}
\end{proposition}
\begin{proof}
  Proving \eqref{eqn:precise_energy_estim_vint_th3} is trivial due to \eqref{eqn:precise_energy_estim_th3}, since
  $\vint_\ire=\phi_i^\epsilon$ on $S_{R\epsilon}^+$. Let
  $\eta=\eta_{R}(\frac\cdot\epsilon)$, with $\eta_R$ defined in \eqref{eqn:def_cut_off};
  then
 \begin{multline*}
  \int_{B_{R\epsilon}^+}\abs{\nabla \vint_\ire}^2\dx \leq \int_{B_{R\epsilon}^+}\abs{\nabla(\eta\phi_i^\epsilon)}^2\dx \leq \\
  \leq 2\left(\int_{B_{R\epsilon}^+} \abs{\nabla \phi_i^\epsilon}^2+\frac{16}{(R\epsilon)^2}\int_{B_{R\epsilon}^+} \abs{\phi_i^\epsilon}^2\dx \right)\leq \const_\rho \int_{\Omega_{R\epsilon}^\epsilon}\abs{\nabla \phi_i^\epsilon}^2\dx,
 \end{multline*}
 where the last step comes from \eqref{eqn:poincare_type_1}. Combining
 this inequality with \eqref{eqn:precise_energy_estim_th1} we obtain
 \eqref{eqn:precise_energy_estim_vint_th1}. Moreover
 \eqref{eqn:poincare_type_1} and
 \eqref{eqn:precise_energy_estim_vint_th1} yield
 \eqref{eqn:precise_energy_estim_vint_th2}. Finally estimates
 \eqref{eq:4}--\eqref{eq:6} follow from the above argument and
 Proposition \ref{prop:energy_estim_eps}.
\end{proof}

Now let us define, for all $i\in\{1,\dots,j\}$, for all $R>1$ and $\epsilon\in (0,1]$ such that $R\epsilon\leq R_{\textup{max}}$,
\begin{equation}\label{eqn:def_vjre}
  v_{i,R,\epsilon}:=\left\{\begin{aligned}
			   & \vint_\ire, && \text{in }B_{R\epsilon}^+, \\
			   & \phi_i^\epsilon, && \text{in }\Omega\setminus B_{R\epsilon}^+,
			 \end{aligned}\right.
\end{equation}
and 
\begin{equation}\label{eqn:def_phi_tilde_Z_R}
 Z_R^\epsilon(x):=\frac{\vint_\jre(\epsilon x)}{\sqrt{H(\phi_j^\epsilon,K_\rho\epsilon)}},\qquad \tilde{\phi}^\epsilon(x):=\frac{\phi_j^\epsilon(\epsilon x)}{\sqrt{H(\phi_j^\epsilon,K_\rho\epsilon)}}.
\end{equation}
 It is easy to prove that the family of functions $\{v_{1,R,\epsilon},\dots,v_\jre\}$ is linearly independent in $H^1_0(\Omega)$.

\begin{lemma}\label{lemma:estim_v}
 For all $R\geq \max\{2,K_\rho\}$, we have that, as $\epsilon\to 0^+$,
 \begin{align}
  \int_\Omega\abs{\nabla v_\jre}^2\dx&=\lambda_j^\epsilon+\epsilon^{N-2}H(\phi_j^\epsilon,K_\rho\epsilon)\left(\int_{B_R^+}\abs{\nabla Z_R^\epsilon}^2\dx-\int_{\Pi_R}\abs{\nabla \tilde{\phi}^\epsilon}^2\dx \right), \label{eqn:estim_vire_1}\\
  \int_\Omega\abs{\nabla v_\ire}^2\dx&= \lambda_i^\epsilon+O(\epsilon^{N-2\rho})\quad\text{for all }i\in\{1,\dots,j\}, \label{eqn:estim_vire_2}\\
  \int_\Omega\nabla v_\ire\cdot\nabla v_\jre\dx&=O\left(\epsilon^{N-1-\rho}\sqrt{H(\phi_j^\epsilon,K_\rho\epsilon)}\right) \quad\text{for all }i\in\{1,\dots,j-1\},\label{eqn:estim_vire_3}\\
  \int_\Omega\nabla v_\ire\cdot\nabla v_\nre\dx&=O(\epsilon^{N-2\rho}) \quad\text{for all }i,n\in\{1,\dots,j\},~ i\neq n, \label{eqn:estim_vire_4}\\
  \int_\Omega p\abs{v_\jre}^2\dx&=1+O(\epsilon^{N-2/N}H(\phi_j^\epsilon,K_\rho\epsilon)), \label{eqn:estim_vire_5}\\
  \int_\Omega p\abs{v_\ire}^2\dx&=1+O(\epsilon^{N+2-2\rho-2/N}) \quad\text{for all }i\in\{1,\dots,j\}, \label{eqn:estim_vire_6}\\
  \int_\Omega p v_\ire v_\jre\dx&=O\left(\epsilon^{N+1-\rho-2/N}\sqrt{H(\phi_j^\epsilon,K_\rho\epsilon)}\right) \quad\text{for all }i\in\{1,\dots,j-1\},\label{eqn:estim_vire_7}\\
  \int_\Omega p v_\ire v_\nre \dx&= O(\epsilon^{N+2-2\rho-2/N})\quad\text{for all }i,n\in\{1,\dots,j\},~ i\neq n, \label{lemma:estim_vire_8}
 \end{align}
 where, in \eqref{eqn:estim_vire_1}, $\tilde{\phi}^\epsilon$ has been trivially extended in $\Pi_R$ outside its domain.
\end{lemma}
\begin{proof}
 We will only prove the first part of the estimates, i.e. \eqref{eqn:estim_vire_1}, \eqref{eqn:estim_vire_2}, \eqref{eqn:estim_vire_3}, \eqref{eqn:estim_vire_4}, since the second part is completely analogous. To prove \eqref{eqn:estim_vire_1} we observe that, by scaling,
 \begin{gather*}
  \int_\Omega\abs{\nabla v_\jre}^2\dx=\int_{B_{R\epsilon}^+}\abs{\nabla \vint_\jre}^2\dx+\int_{\Omega^\epsilon}\abs{\nabla \phi_j^\epsilon}^2\dx-\int_{\Omega_{R\epsilon}^\epsilon}\abs{\nabla\phi_j^\epsilon}^2\dx \\
  =\lambda_j^\epsilon+\epsilon^{N-2}H(\phi_j^\epsilon,K_\rho\epsilon)\left(\int_{B_R^+}\abs{\nabla Z_R^\epsilon}^2\dx-\int_{\Pi_R}\abs{\nabla \tilde{\phi}^\epsilon}^2\dx \right).
 \end{gather*}
  Thanks to Propositions \ref{prop:energy_estim_eps} and \ref{prop:precise_energy_estim_vint} we have that
 \begin{gather*}
  \int_\Omega\abs{\nabla v_\ire}^2\dx=\int_{B_{R\epsilon}^+}\abs{\nabla \vint_\ire}^2\dx+\int_{\Omega^\epsilon}\abs{\nabla \phi_i^\epsilon}^2\dx-\int_{\Omega_{R\epsilon}^\epsilon}\abs{\nabla\phi_i^\epsilon}^2\dx\\
  =\lambda_i^\epsilon+O(\epsilon^{N-2\rho}),
 \end{gather*}
as $\epsilon\to0^+$, thus proving \eqref{eqn:estim_vire_2} and, by Cauchy-Schwarz Inequality, for $i<j$
 \begin{gather*}
 \int_\Omega\nabla v_\ire\cdot\nabla v_\jre=\int_{B_{R\epsilon}^+}\nabla \vint_\ire\nabla \vint_\jre\dx+\int_{\Omega^\epsilon}\nabla \phi_i^\epsilon \cdot\nabla \phi_j^\epsilon\dx-\int_{\Omega_{R\epsilon}^\epsilon}\nabla \phi_i^\epsilon\cdot\nabla\phi_j^\epsilon\dx \\
 =O(\epsilon^{\frac{N-2\rho}{2}})O\left(\epsilon^{\frac{N-2}{2}}\sqrt{H(\phi_j^\epsilon,K_\rho\epsilon)}\right)=O\left(\epsilon^{N-1-\rho}\sqrt{H(\phi_j^\epsilon,K_\rho\epsilon)}\right)  ,
 \end{gather*}
as $\epsilon\to0^+$, thus proving \eqref{eqn:estim_vire_3}. Similarly, for $i\neq n$
  \begin{gather*}
 \int_\Omega\nabla v_\ire\cdot\nabla v_\nre=\int_{B_{R\epsilon}^+}\nabla \vint_\ire\nabla \vint_\nre\dx+\int_{\Omega^\epsilon}\nabla \phi_i^\epsilon \cdot\nabla \phi_n^\epsilon\dx-\int_{\Omega_{R\epsilon}^\epsilon}\nabla \phi_i^\epsilon\cdot\nabla\phi_n^\epsilon\dx \\
 =O(\epsilon^{N-2\rho})  ,
 \end{gather*}
as $\epsilon\to0^+$, which provides \eqref{eqn:estim_vire_4}.
\end{proof}

We construct a basis $\{\hat v_{1,R,\epsilon},\dots,\hat v_\jre\}$ of
the space $\textrm{span}\,\{v_{1,R,\epsilon},\dots,v_\jre\}$ such that
\[
\int_\Omega p\, \hat v_{n,R,\epsilon}\hat v_{m,R,\epsilon}\,dx=0\quad\text{for
$n\neq m$},
\]
by defining
 \begin{gather*}
  \hat{v}_\jre = v_\jre,\quad
  \hat{v}_\ire=v_\ire-\sum_{n=i+1}^j d_{i,n}^\epsilon \hat{v}_\nre,\quad\text{for all $i=1,\dots,j-1$ },
 \end{gather*}
 where
 \[
  d_{i,n}^\epsilon=\frac{\int_{\Omega}p\,v_\ire\,\hat{v}_\nre\dx }{\int_{\Omega}p\abs{\hat{v}_\nre}^2\dx}.
 \]
Using the estimates established in Lemma \ref{lemma:estim_v}, one can prove the following
 \begin{align}
  \int_{\Omega}\abs{\nabla \hat{v}_\jre }^2\dx &= \lambda_j^\epsilon+\epsilon^{N-2}H(\phi_j^\epsilon,K_\rho\epsilon)\left(\int_{B_R^+}\abs{\nabla Z_R^\epsilon}^2\dx-\int_{\Pi_R}\abs{\nabla \tilde{\phi}^\epsilon}^2\dx \right),\label{eqn:v_graham_1}\\
  \int_{\Omega}\abs{\nabla \hat{v}_\ire }^2\dx &=\lambda_i^\epsilon+O(\epsilon^{N-2\rho}) \quad\text{for all }i\in\{1,\dots,j\}, \label{eqn:v_graham_2}\\
  \int_{\Omega}\nabla \hat{v}_\jre\cdot \nabla \hat{v}_\ire \dx &=O(\epsilon^{N-1-\rho} \sqrt{H(\phi_j^\epsilon,K_\rho\epsilon)}) \quad\text{for all }i\in\{1,\dots,j-1\}, \label{eqn:v_graham_3}\\
  \int_{\Omega}\nabla \hat{v}_\ire\cdot \nabla \hat{v}_\mre \dx &=O(\epsilon^{N-2\rho}) \quad\text{for all }i,m\in\{1,\dots,j\}, \quad i\neq m ,\label{eqn:v_graham_4}\\
  \int_{\Omega}p\abs{ \hat{v}_\jre }^2\dx &=1+O(\epsilon^{N-2/N} H(\phi_j^\epsilon,K_\rho\epsilon)), \label{eqn:v_graham_5} \\
  \int_{\Omega}p\abs{ \hat{v}_\ire }^2\dx &=1+O(\epsilon^{N+2-2\rho-2/N}) \quad\text{for all }i\in\{1,\dots,j\} ,\label{eqn:v_graham_6}
 \end{align}
 as $\epsilon\to 0$.

\begin{proposition}\label{prop:up_bound}
  Let $\rho\in(0,1/2)$, $K_\rho$ as defined in Proposition 
\ref{prop:precise_energy_estim_eps} and $R\geq K_\rho$. For $\epsilon <R_0/R$ there exists $f_R(\epsilon)$ such that
  \begin{equation*}
    \lambda_j-\lambda_j^\epsilon\leq \epsilon^{N-2}H(\phi_j^\epsilon,K_\rho\epsilon)(f_R(\epsilon)+o(1))\quad\text{as }\epsilon\to 0
  \end{equation*}
  and
  \[
    f_R(\epsilon)=\int_{B_R^+}\abs{\nabla Z_R^\epsilon}^2\dx-\int_{\Pi_R}\abs{\nabla \tilde{\phi}^\epsilon}^2\dx,
  \]
  where $\tilde{\phi}^\epsilon$ has been trivially extended in $\Pi_R$ outside its domain.
\end{proposition}
\begin{proof}
  By the \emph{Courant-Fischer Min-Max} characterization of eigenvalues
  \[
  \lambda_j=\min \left\{ \max_{\substack{\alpha_1,\dots,\alpha_j\in \R
        \\ \sum_{i=1}^j\abs{\alpha_i}^2=1}}\frac{\int_{\Omega}\abs{\nabla
        \left(\sum_{i=1}^j \alpha_i u_i\right)}^2\dx}{\int_{\Omega}p\abs{\sum_{i=1}^j
        \alpha_i u_i}^2\dx} \colon \{u_1,\dots,u_j\}\subseteq
    H^1_0(\Omega) ~ \text{linearly independent}\right\}.
  \]
  Testing the  Rayleigh  quotient with the family of functions
  \[
    \tilde{v}_\ire=\frac{\hat{v}_\ire}{\sqrt{\int_{\Omega}p\abs{\hat{v}_\ire}^2\dx}}
  \]
  we obtain that
  \begin{gather*}
    \lambda_j-\lambda_j^\epsilon\leq
    \max_{\substack{\alpha_1,\dots,\alpha_j\in \R \\ 
\sum_{i=1}^j\abs{\alpha_i}^2=1}}\int_{\Omega}\bigg|\nabla \bigg(\sum_{i=1}^j \alpha_i \tilde{v}_\ire\bigg)\bigg|^2\dx-\lambda_j^\epsilon
    =\max_{\substack{\alpha_1,\dots,\alpha_j\in \R \\
        \sum_{i=1}^j\abs{\alpha_i}^2=1}}\sum_{i,n=1}^j
    M_{i,n}^\epsilon\alpha_i\alpha_n
  \end{gather*}
  where
  \[
    M_{i,n}^\epsilon=\frac{\int_{\Omega}\nabla\hat{v}_\ire\cdot\nabla\hat{v}_\nre\dx}{\left(\int_{\Omega}p\abs{\hat{v}_\ire}^2\dx\right)^{1/2}\left(\int_{\Omega}p\abs{\hat{v}_\nre}^2\dx\right)^{1/2}}-\lambda_j^\epsilon\delta_i^n,
  \]
 with $\delta_i^n$ denoting the usual \emph{Kronecker delta}, i.e. 
$\delta_i^n=0$ for $i\neq n$ and $\delta_i^n=1$ for $i=n$. From
estimates \eqref{eqn:v_graham_1}--\eqref{eqn:v_graham_6} one can
derive the following estimates
\begin{align*}
  &  M_{j,j}(\epsilon)=\epsilon^{N-2}H(\phi_j^\epsilon,K_\rho\epsilon)(f_R(\epsilon)+O(\epsilon^{2-2/N})), \\
&    M_{i,j}(\epsilon)=O\left(\epsilon^{N-1-\rho}\sqrt{H(\phi_j^\epsilon,K_\rho\epsilon)}\right)
\quad\text{and}\quad
M_{i,i}(\epsilon)=\lambda_i^\epsilon-\lambda_j^\epsilon+o(1)  
\quad\text{for all }i<j, \\
  & M_{i,n}(\epsilon)=O(\epsilon^{N-2\rho})\quad\text{for all }i,n<j,\quad i\neq n,
\end{align*}
  as $\epsilon\to 0$. Moreover, from Corollary \ref{cor:H_control_below}, we know that
  $H(\phi_j^\epsilon,K_\rho\epsilon)\geq \bar{C}\epsilon^q$ 
  for some $\bar{C},q>0$. Therefore, taking also into account
  \eqref{eqn:H_bigO} and the fact that $f_R(\epsilon)=O(1)$ as
  $\epsilon\to0$ in view of \eqref{eqn:precise_energy_estim_th1} and \eqref{eqn:precise_energy_estim_vint_th1},
the hypotheses of Lemma \ref{lemma:quadratic_form} are satisfied with
  \[
    \sigma(\epsilon)=\epsilon^{N-2}H(\phi_j^\epsilon,K_\rho\epsilon),\quad \mu(\epsilon)=f_R(\epsilon)+o(1),\quad \alpha=\frac{N}{2}-\rho,\quad M>(2\rho-2+q)\frac{2}{N-2\rho}.
  \]
  The proof is thereby complete.
\end{proof}

\subsection{Lower Bound}

For any $R>1$ and $\epsilon\in (0,1]$, with $R\epsilon\leq R_{\textup{max}}$, let us consider the following minimization problem
\begin{equation}\label{eqn:min_wint}
 \min \left\{\int_{\Omega_{R\epsilon}^\epsilon}\abs{\nabla u}^2\dx\colon u\in H^1(\Omega_{R\epsilon}^\epsilon),~u=0~\text{on}~\partial\Omega_{R\epsilon}^\epsilon\setminus S_{R\epsilon}^+,~u=\phi_j~\text{on}~S_{R\epsilon}^+\right\}.
\end{equation}
One can prove that this problem has a unique solution $w^{\textup{int}}_{j,R,\epsilon}$, which weakly verifies
\begin{equation*}
 \left\{\begin{aligned}
         -\Delta \wint_\jre & =0 ,&& \text{in }\Omega_{R\epsilon}^\epsilon, \\
         \wint_\jre & = 0, && \text{on }\partial\Omega_{R\epsilon}^\epsilon\setminus S_{R\epsilon}^+, \\
         \wint_\jre & = \phi_j, && \text{on }S_{R\epsilon}^+.
        \end{aligned}\right.
\end{equation*}
Let us define
\begin{equation*}
 w_{j,R,\epsilon}:=\left\{\begin{aligned}
			   & \wint_\jre ,&& \text{in }\Omega_{R\epsilon}^\epsilon, \\
			   & \phi_j, && \text{in }\Omega\setminus B_{R\epsilon}^+.
			 \end{aligned}\right.
\end{equation*}

\begin{lemma}\label{lemma:energy_estim_zero}
 There exists $\tilde{C}>0$ such that, for all $i\in \{1,\dots,j-1\}$, for all $R>1$ and $\epsilon\in (0,1]$, with $R\epsilon \leq R_{\textup{max}}$, 
 \begin{gather*}
  \int_{B_{R\epsilon}^+}\abs{\nabla \phi_i}^2\dx\leq \tilde{C}(R\epsilon)^N, \quad
  \int_{B_{R\epsilon}^+}\abs{ \phi_i}^2\dx\leq \tilde{C}(R\epsilon)^{N+2}, \quad
  \int_{S_{R\epsilon}^+}\abs{ \phi_i}^2\dx\leq \tilde{C}(R\epsilon)^{N+1},
 \end{gather*}
 and
 \begin{gather*}
    \int_{B_{R\epsilon}^+}\abs{\nabla \phi_j}^2\dx\leq \tilde{C}(R\epsilon)^{N+2k-2}, \quad
  \int_{B_{R\epsilon}^+}\abs{ \phi_j}^2\dx\leq \tilde{C}(R\epsilon)^{N+2k}, \quad
  \int_{S_{R\epsilon}^+}\abs{ \phi_j}^2\dx\leq \tilde{C}(R\epsilon)^{N+2k-1}.  
 \end{gather*}

\end{lemma}
\begin{proof}
 It follows from classical asymptotic estimates at the boundary, see e.g. \cite[Th. 1.3]{Felli2011} and \eqref{eq:1},\eqref{eq:2}.
\end{proof}

\begin{lemma}\label{lemma:energy_estim_wint}
 There exists $\hat{C}>0$ such that, for all $R>1$ and $\epsilon\in(0,1]$, with $R\epsilon\leq R_{\textup{max}}$,
  \begin{gather}
  \int_{\Omega_{R\epsilon}^\epsilon}\abs{\nabla \wint_\jre}^2\dx\leq \hat{C}(R\epsilon)^{N+2k-2}, \label{eqn:wint_estim_1}\\
  \int_{S_{R\epsilon}^+}\abs{ \wint_\jre}^2\dx\leq \hat{C}(R\epsilon)^{N+2k-1}. \label{eqn:wint_estim_3}
 \end{gather}
Furthermore, for all $R>\mu_{1/2}$ and $\epsilon\in(0,1]$, with $R\epsilon\leq R_{\textup{max}}$,
\begin{equation}
 \int_{\Omega_{R\epsilon}^\epsilon}\abs{ \wint_\jre}^2\dx\leq \hat{C}(R\epsilon)^{N+2k-2/N}. \label{eqn:wint_estim_2}
\end{equation}
\end{lemma}
\begin{proof}
 \eqref{eqn:wint_estim_3} is trivial since $\wint_\jre=\phi_j$ on $S_{R\epsilon}^+$. Also \eqref{eqn:wint_estim_1} is simple since $\phi_j$ is an admissible test function for \eqref{eqn:min_wint}. Finally \eqref{eqn:wint_estim_2} comes from \eqref{eqn:wint_estim_1}, \eqref{eqn:wint_estim_3} and Lemma \ref{lemma:estim_1}.
\end{proof}

Now let us define, for all $R>1$ and $\epsilon\in (0,1]$ such that $R\epsilon\leq R_{\textup{max}}$,
\begin{equation}\label{eqn:def_W_U_R}
 U_R^\epsilon(x)=\frac{\wint_\jre(\epsilon x)}{\epsilon^k},\qquad W^\epsilon(x)=\frac{\phi_j(\epsilon x)}{\epsilon^k}.
\end{equation}
 From \eqref{eq:1}, we easily deduce that 
 \[
  W^\epsilon\longrightarrow \psi_k\qquad\text{in }H^1(B_R^+)\quad\text{as }\epsilon\to 0,\quad\text{for all }R>0,
 \]
 where $\psi_k$ has been defined in \eqref{eqn:def_psi_k}.

\begin{lemma}\label{lemma:conv_U_R_eps}
 We have that
 \[
  U_R^\epsilon\longrightarrow U_R \qquad\text{in }\mathcal{H}_R \quad \text{as }\epsilon\to 0,\quad\text{for all } R>1,
 \]
 where $U_R$ is defined in Lemma \ref{lemma:def_U_R}.
\end{lemma}
\begin{proof}
 From Lemma \ref{lemma:energy_estim_wint} and from the definition of $U_R^\epsilon$ we know that
 \[
  \int_{\Pi_R}\abs{\nabla U_R^\epsilon}^2\dx=O(1) \qquad\text{as }\epsilon\to 0.
 \]
 where $U_R^\epsilon$ has been trivially extended in $\Pi_R$ outside
 its domain. So there exists $V=V_R\in \mathcal{H}_R$ such that, along
 a sequence $\epsilon=\epsilon_n\to0$,
 \[
  U_R^\epsilon\rightharpoonup V\qquad\text{weakly in }\mathcal{H}_R\quad\text{as }\epsilon=\epsilon_n\to 0.
 \]
 This means that
 \[
  \nabla U_R^\epsilon \rightharpoonup \nabla V\qquad\text{weakly in }L^2(\Pi_R)\quad\text{as }\epsilon=\epsilon_n\to 0.
 \]
 Since $U_R^\epsilon=W^\epsilon$ on $S_R^+$ and $W^\epsilon\to \psi_k$
 in $L^2(S_R^+)$, then $V$ satisfies (in a weak sense) the same
 equation as $U_R$, defined in Lemma \ref{lemma:def_U_R}. So, by
 uniqueness, $V=U_R$. 
Since the limit $V=U_R$ is the same for every subsequence,
\emph{Urysohn's Subsequence Principle} implies that the convergence
$U_R^\epsilon\rightharpoonup U_R$ holds as $\epsilon \to 0$ (not
only along subsequences).

To prove strong convergence it is enough to show that
$\|U_R^\epsilon\|_{\mathcal H_R}\to \|U_R\|_{\mathcal H_R}$ as $\epsilon\to0$. First we notice that, trivially, $-\Delta U_R^\epsilon\rightharpoonup -\Delta U_R$ weakly in $L^2(\Pi_R)$: so, we have that $\nabla U_R^\epsilon\rightharpoonup \nabla U_R$ in $H(\div,\Pi_R)$, thus
 \[
  \frac{\partial U_R^\epsilon}{\partial \nnu}\rightharpoonup \frac{\partial U_R}{\partial\nnu} \qquad\text{in }\left( H^{1/2}_{00}(S_R^+) \right)^*\quad\text{as }\epsilon\to 0,
 \]
 where $( H^{1/2}_{00}(S_R^+) )^*$ is the dual of the Lions-Magenes space $H^{1/2}_{00}(S_R^+)$. Then, since $W^\epsilon\to \psi_k$ in $H^{1/2}_{00}(S_R^+)$ as $\epsilon\to 0$, we obtain that
 \[
  \int_{\Pi_R}\abs{\nabla
    U_R^\epsilon}^2\dx=\int_{S_R^+}\frac{\partial
    U_R^\epsilon}{\partial\nnu}W^\epsilon\ds\to
  \int_{S_R^+}\frac{\partial
    U_R}{\partial\nnu}\psi_k\ds=\int_{\Pi_R}\abs{\nabla
    U_R}^2\quad\text{as }\epsilon\to0,
 \]
thus completing the proof.
\end{proof}

It is easy to prove that the family of functions
$\{\phi_1,\phi_2,\dots,\phi_{j-1}, w_\jre\}$ is linearly independent
in $ H^1_0(\Omega^\epsilon)$. As in the previous section, we construct
a new basis of the space
\[
  \textrm{span}\{\phi_1,\phi_2,\dots,\phi_{j-1}, w_\jre\}\subseteq H^1_0(\Omega^\epsilon)
\]
by defining, for all $i=1,\dots,j-1$
 \[
  \hat{w}_\ire=\phi_i
 \]
 and
 \[
   \hat{w}_\jre=w_\jre -\sum_{i=1}^{j-1}c_{i}^\epsilon \phi_i,
 \]
 where
 \[
   c_{i}^\epsilon=\int_{\Omega^\epsilon}p w_\jre \phi_i\dx.
 \]
 In this way we have that $\int_{\Omega^\epsilon}p\,
 \hat{w}_{n,R,\epsilon} \hat{w}_{m,R,\epsilon}\,dx=0$ if $n\neq m$.

Using the estimates established in Lemmas \ref{lemma:energy_estim_zero} and \ref{lemma:energy_estim_wint}, one can prove the following
 \begin{align}
  \int_{\Omega^\epsilon}\abs{\nabla \hat{w}_\jre }^2\dx &= \lambda_j+\epsilon^{N+2k-2}\left(\int_{\Pi_R}\abs{\nabla U_R^\epsilon}^2\dx -\int_{B_R^+}\abs{\nabla W^\epsilon}^2\dx+o(1)\right),\label{eqn:w_graham_1}\\
  \int_{\Omega^\epsilon}\nabla \hat{w}_\jre\cdot \nabla \hat{w}_\ire \dx &=O(\epsilon^{N+k-1}) \quad\text{for all }i\in \{1,\dots,j-1\} ,\label{eqn:w_graham_3}\\
  \int_{\Omega^\epsilon}p\abs{ \hat{w}_\jre }^2\dx &=1+O(\epsilon^{N+2k-2/N}), \label{eqn:w_graham_5} 
 \end{align}
 as $\epsilon\to 0$.

\begin{proposition}\label{prop:low_bound}
  Let $\rho\in(0,1/2)$, $K_\rho$ as defined in Proposition \ref{prop:precise_energy_estim_eps} and $R\geq K_\rho$. For $\epsilon <R_0/R$ there exists $g_R(\epsilon)$ such that
  \begin{equation*}
    \lambda_j^\epsilon-\lambda_j\leq \epsilon^{N+2k-2}(g_R(\epsilon)+o(1))\quad\text{as }\epsilon\to 0
  \end{equation*}
  and
  \[
    g_R(\epsilon)=\int_{\Pi_R}\abs{\nabla U_R^\epsilon}^2\dx-\int_{B_R^+}\abs{\nabla W^\epsilon}^2\dx,
  \]
  where $U_R^\epsilon$ has been trivially extended in $\Pi_R$ outside its domain.
\end{proposition}
\begin{proof}
  By the \emph{Courant-Fischer Min-Max} characterization of eigenvalues
  \[
    \lambda_j^\epsilon=\min\left\{
      \max_{\substack{\alpha_1,\dots,\alpha_j\in \R \\
          \sum_{i=1}^j\abs{\alpha_i}^2=1}}\frac{\int_{\Omega^\epsilon}\big|\nabla
        \big(\sum_{i=1}^j \alpha_i u_i\big)\big|^2\dx}
{\int_{\Omega^\epsilon}p\big|\sum_{i=1}^j \alpha_i u_i\big|^2\dx} \colon\{ u_1,\dots,u_j\}\subseteq  H^1_0(\Omega^\epsilon) ~\text{linearly independent}\right\}.
  \]
  Testing the  Rayleigh quotient with the family of functions
  \[
    \tilde{w}_\ire=\frac{\hat{w}_\ire}{\sqrt{\int_{\Omega^\epsilon}p\abs{\hat{w}_\ire}^2\dx}}
  \]
  we obtain that
 \[
    \lambda_j^\epsilon-\lambda_j\leq \max_{\substack{\alpha_1,\dots,\alpha_j\in \R \\ \sum_{i=1}^j\abs{\alpha_i}^2=1}}\int_{\Omega^\epsilon}\abs{\nabla \left(\sum \alpha_i \tilde{w}_\ire\right)}^2\dx-\lambda_j
    =\max_{\substack{\alpha_1,\dots,\alpha_j\in \R \\ \sum_{i=1}^j\abs{\alpha_i}^2=1}}\sum_{i,n=1}^j L_{i,n}^\epsilon\alpha_i\alpha_n ,
\]
  where
  \[
    L_{i,n}^\epsilon=\frac{\int_{\Omega^\epsilon}\nabla\hat{w}_\ire\cdot\nabla\hat{w}_\nre\dx}{\left(\int_{\Omega^\epsilon}p\abs{\hat{w}_\ire}^2\dx\right)^{1/2}\left(\int_{\Omega^\epsilon}p\abs{\hat{w}_\nre}^2\dx\right)^{1/2}}-\lambda_j\delta_i^n,
  \]
  with  $\delta_i^n$ denoting the usual \emph{Kronecker delta}. 
From estimates \eqref{eqn:w_graham_1}--\eqref{eqn:w_graham_5} it
follows that
  \begin{gather*}
   L_{j,j}^\epsilon=\epsilon^{N+2k-2}(g_R(\epsilon)+o(1)),\quad
   L_{i,j}^\epsilon=O(\epsilon^{N+k-1})\quad\text{for all }i<j, \\
   L_{i,i}^\epsilon=\lambda_i-\lambda_j  \quad\text{for all }i<j, \quad
   L_{i,n}^\epsilon=0\quad\text{for all }i,n<j,\quad i\neq n,
  \end{gather*}
  as $\epsilon\to 0$. Therefore, taking into account that 
$g_R(\epsilon)=O(1)$ as $\epsilon \to 0$ in view of Lemma
\ref{lemma:energy_estim_zero} and \eqref{eqn:wint_estim_1},
the hypotheses of Lemma \ref{lemma:quadratic_form} are satisfied with
  \[
    \sigma(\epsilon)=\epsilon^{N+2k-2},\quad \mu(\epsilon)=g_R(\epsilon)+o(1),\quad \alpha=\frac{N}{2},\quad M>\frac{4(k-1)}{N}.
  \]
  The proof is thereby complete.
  \end{proof}

From the fact that $W^\epsilon\to \psi_k$ in $H^1(B_R^+)$, as
$\epsilon\to 0$, for all $R>0$, and from Lemma
\ref{lemma:conv_U_R_eps} we can deduce the following result.

\begin{lemma}\label{lemma:conv_g_R_eps}
  For all $R>1$ we have that
  \[
    g_R(\epsilon)\longrightarrow g_R\qquad\text{as }\epsilon\to 0,
  \]
  where
  \begin{equation}
        g_R:=\int_{\Pi_R}\abs{\nabla U_R}^2\dx-\int_{B_R^+}\abs{\nabla \psi_k}^2\dx. \label{eqn:def_g_R}
  \end{equation}
\end{lemma}
In order to compute the limit $\lim_{R\to \infty}g_R$, we  introduce
the functions 
    \begin{gather}
   \zeta(r):=\int_{S_1^+}\Phi(r\theta)\Psi(\theta)\dth  \qquad\text{for }r\geq 1, \label{eqn:def_zeta_r} \\
   \chi_R(r):=\int_{S_1^+}U_R(r\theta)\Psi(\theta)\dth \qquad\text{for }1\leq r\leq R. \label{eqn:def_chi_R}
   \end{gather}
  Moreover, we denote
  \begin{equation}\label{eqn:def_gamma_N}
   \gamma_N:=\int_{S_1^+}\abs{\Psi(\theta)}^2\dth.
  \end{equation}
  We immediatly notice, thanks to Lemma \ref{lemma:conv_U_R} and to the embedding $H^1(B_1^+)\hookrightarrow L^2(S_1^+)$, that
  \begin{equation}\label{eqn:zeta_lim_chi}
   \zeta(1)=\lim_{R\to+\infty}\chi_R(1).
  \end{equation}

\begin{lemma}\label{lemma:zeta_1}
  Let $\zeta$ be the function defined in \eqref{eqn:def_zeta_r}, $\gamma_N$ the constant defined in \eqref{eqn:def_gamma_N} and $m_k(\Sigma)$ the one defined in \eqref{eqn:def_m_k}. Then
   \begin{equation}
    \zeta(1)=\gamma_N-\frac{2m_k(\Sigma)}{N+2k-2}\label{eqn:th_lemma_zeta_1}.
   \end{equation}
  \end{lemma}
  \begin{proof}
   From the definition of $\Phi$, given in \eqref{eqn:def_Phi}, one can easily prove that $\zeta$ satisfies the following ODE
   \begin{equation*}
    \left( r^{N+2k-1}(r^{-k}\zeta(r))' \right)'=0\qquad\text{in }(1,+\infty).
   \end{equation*}
   This yields
   \begin{equation}\label{eqn:zeta_r_1}
    r^{-k}\zeta(r)=\zeta(1)+C\frac{1-r^{-N-2k+2}}{N+2k-2}
   \end{equation}
for some constant $C\in\R$.
   Now we note that $r^{-k}\zeta(r)\to \gamma_N$ as $r\to+\infty$. Indeed, since $\Phi=w_k+\psi_k$, we can rewrite
   \[
    \zeta(r)=\int_{S_1^+}w_k(r\theta)\Psi(\theta)\dth+\gamma_N r^k.
   \]
   By evaluating the vanishing order at $0$ of the Kelvin transform of
   the restriction of the function $w_k$ on $\Pi\setminus\Pi_1$, we
   can prove that
   \[
    \abs{w_k(x)}\leq \const \abs{x}^{1-N}\qquad\text{for }\abs{x}>1.
   \]
   Hence, when $r\to +\infty$
   \[
    \abs{r^{-k}\zeta(r)-\gamma_N}\leq \int_{S_1^+}\frac{\abs{w_k(r\theta)}}{r^k}\abs{\Psi(\theta)}\dth\leq \const r^{1-N-k}\to 0.
   \]
   Then we can find the constant $C$ in \eqref{eqn:zeta_r_1}, letting
   $r\to +\infty$; so we can rewrite $\zeta$ as
   \begin{equation}\label{eqn:zeta_r_7}
    \zeta(r)=\gamma_N r^k+(\zeta(1)-\gamma_N)r^{-N-k+2} \qquad\text{in }(1,+\infty).
   \end{equation}
   Taking the derivative leads to 
   \begin{equation}\label{eqn:zeta'_r}
   \begin{aligned}
    \zeta'(r)&=k\gamma_N r^{k-1}+(N+k-2)(\gamma_N-\zeta(1))r^{-N-k+1} \\
    & =(N+2k-2)\gamma_N r^{k-1}-\frac{(N+k-2)\zeta(r)}{r}.
   \end{aligned}
   \end{equation}
   Hence, taking into account the definition of $\zeta$ and evaluating its derivative at $r=1$, we obtain
   \begin{equation}\label{eqn:zeta_r_2}
    \int_{S_1^+}\frac{\partial\Phi}{\partial\nnu}(\theta)\Psi(\theta)\dth=(N+2k-2)\gamma_N-(N+k-2)\zeta(1).
   \end{equation}
   Since $-\Delta\Phi=0$ in $B_1^+$, multiplying this equation by
   $\psi_k$ and integrating by parts we obtain that 
   \begin{equation}\label{eqn:zeta_r_3}
        \int_{B_1^+}\nabla\Phi\cdot \nabla \psi_k\dx=\int_{S_1^+}\frac{\partial\Phi}{\partial\nnu}\psi_k\ds=\int_{S_1^+}\frac{\partial\Phi}{\partial\nnu}\Psi\ds.
   \end{equation}
   Then, let us test the equation $-\Delta \psi_k=0$ with $\Phi$. From
   \eqref{eqn:m_k_integral} and \eqref{eqn:Phi_components} it follows that
   \begin{equation}\label{eqn:zeta_r_4}
   \begin{split}
    \int_{B_1^+}\nabla \psi_k\cdot\nabla \Phi\dx=\int_{S_1^+}\frac{\partial\psi_k}{\partial\nnu}\Phi\ds-\int_{\Sigma}\frac{\partial\psi_k}{\partial x_1}\Phi\ds &=\int_{S_1^+}\frac{\partial\psi_k}{\partial\nnu}\Phi\ds-\int_{\Sigma}\frac{\partial\psi_k}{\partial x_1}w_k\ds \\&=\int_{S_1^+}\frac{\partial\psi_k}{\partial\nnu}\Phi\ds+2m_k(\Sigma).
   \end{split}
   \end{equation}
   Moreover we note that
   \begin{equation}\label{eqn:zeta_r_5}
        \frac{\partial\psi_k}{\partial\nnu}(\theta)=k\Psi(\theta)\qquad\text{on }S_1^+.
   \end{equation}
   Then, from \eqref{eqn:zeta_r_4} and \eqref{eqn:zeta_r_5} we obtain
   \begin{equation}\label{eqn:zeta_r_6}
    \int_{B_1^+}\nabla \psi_k\cdot\nabla \Phi\dx=k\int_{S_1^+}\Phi\Psi\ds+2m_k(\Sigma)=k\zeta(1)+2m_k(\Sigma).
   \end{equation} 
   Finally, combining \eqref{eqn:zeta_r_2}, \eqref{eqn:zeta_r_3} and \eqref{eqn:zeta_r_6} leads to the thesis. 
     \end{proof}
  
  \begin{lemma}\label{lemma:g_R}
  Let $g_R$ be as defined in \eqref{eqn:def_g_R} and $m_k(\Sigma)$ as
  in \eqref{eqn:def_m_k}. Then $\lim_{R\to+\infty}g_R=2m_k(\Sigma)$.
  \end{lemma}
  \begin{proof}
     Integrating by parts we have that
  \[
   g_R=\int_{S_R^+}\left( \frac{\partial U_R}{\partial \nnu}-\frac{\partial \psi_k}{\partial\nnu} \right)\psi_k\ds.
  \]
  If $\chi_R$ is the function defined in \eqref{eqn:def_chi_R}, then
  \[
   \chi_R'(r)=\int_{S_1^+}\frac{\partial U_R}{\partial\nnu}(r\theta)\Psi(\theta)\dth
  \]
  and, by a change of variable,
  \begin{equation}\label{eqn:chi'_R}
      \chi_R'(r)=r^{1-N-k}\int_{S_r^+}\frac{\partial U_R}{\partial\nnu}\psi_k\ds.
  \end{equation}
  By simple computations one can prove that $\chi_R$ solves
  \[
   \left( r^{N+2k-1}(r^{-k}\chi_R(r))' \right)'=0\qquad\text{in }(1,R).
  \]
  By integration, we arrive at
  \begin{equation}\label{eqn:chi_R_1}
   r^{-k}\chi_R(r)=\chi_R(1)+C\frac{1-r^{-N-2k+2}}{N+2k-2}.
  \end{equation}
  From the fact that $U_R=R^K\Psi$ on $S_R^+$, we have that $\chi_R(R)=R^k\gamma_N$, and this allows us to know the constant $C$. After some computations, the expression \eqref{eqn:chi_R_1} then becomes as follows
  \begin{equation*}
   r^{-k}\chi_R(r)=\chi_R(1)+(\gamma_N-\chi_R(1))\frac{1-r^{-N-2k+2}}{1-R^{-N-2k+2}},\qquad r\in (1,R).
  \end{equation*}
  From \eqref{eqn:chi'_R}, we get
  \begin{equation}\label{eqn:chi_R_2}
  \begin{split}
   \int_{S_R^+}\frac{\partial U_R}{\partial\nnu}\psi_k\ds&=\chi_R'(R)R^{N+k-1} \\
   & =\frac{[\gamma_N(N+k-2)-\chi_R(1)(N+2k-2)]R^{-N-k+1}+k\gamma_N R^{k-1}}{1-R^{-N-2k+2}}R^{N+k-1} \\
   & = \frac{\gamma_N(N+k-2)-\chi_R(1)(N+2k-2)+k\gamma_N R^{N+2k-2}}{ 1-R^{-N-2k+2}}.
   \end{split}
  \end{equation}
  For what concerns the second part of $g_R$, it is easy to see that
  \begin{equation*}
   \frac{\partial \psi_k}{\partial\nnu}(r\theta)=kr^{k-1}\Psi(\theta).
  \end{equation*}
  Therefore
  \begin{equation}\label{eqn:chi_R_3}
   \int_{S_R^+}\frac{\partial\psi_k}{\partial\nnu}\psi_k\ds=\int_{S_1^+}kR^{N+2k-2}\abs{\Psi}^2\dth=k R^{N+2k-2}\gamma_N.
  \end{equation}
  Finally, combining \eqref{eqn:zeta_lim_chi}, \eqref{eqn:chi_R_2}, \eqref{eqn:chi_R_3} and Lemma \ref{lemma:zeta_1} and taking the limit when $R\to +\infty$ we reach the conclusion.
  \end{proof}

  Combining Propositions \ref{prop:up_bound} and \ref{prop:low_bound}
  with Lemmas \ref{lemma:conv_g_R_eps} and \ref{lemma:g_R} we obtain the following upper-lower
  estimate for the eigenvalue variation.

\begin{proposition}\label{prop:up_low_bound}
    Let $\rho\in(0,1/2)$, $K_\rho$ as defined in Proposition \ref{prop:precise_energy_estim_eps} and $m_k(\Sigma)$ as in \eqref{eqn:def_m_k}. Then, for all $R\geq K_\rho$, we have that, as $\epsilon\to 0$,
    \[
      -2m_k(\Sigma)+o(1)\leq \frac{\lambda_j-\lambda_j^\epsilon}{\epsilon^{N+2k-2}}\leq \frac{H(\phi_j^\epsilon,K_\rho\epsilon)}{\epsilon^{2k}}(f_R(\epsilon)+o(1)).
    \]
\end{proposition}

Since $-2m_k(\Sigma)>0$, as a direct consequence of Proposition
\ref{prop:up_low_bound} we obtain the following estimate from below 
for $H(\phi_j^\epsilon,K_\rho\epsilon)$.

\begin{corollary}\label{cor:bdd_blow_up}
We have that
\begin{equation*}
\frac{\epsilon^{2k}}{H(\phi_j^\epsilon,K_\rho\epsilon)}=O(1)\qquad\text{as }\epsilon\to 0.
\end{equation*}
\end{corollary}

\section{Blow-up Analysis}\label{sec:blow-up-analysis}

Let us introduce the functional

\begin{align*}
  F\colon \R \times H^1_0(\Omega) & \longrightarrow \R\times H^{-1}(\Omega) \\
  (\lambda,\phi) &\longmapsto (\norm{\phi}_{H^1_0(\Omega)}^2-\lambda_j,-\Delta \phi-\lambda p\phi)
\end{align*}
where $\|\phi\|_{H^1_0(\Omega)}^2=\int_\Omega |\nabla\phi|^2\,dx$ and 
\[
  _{H^{-1}(\Omega)}\<-\Delta\phi-\lambda p\phi, v\>_{H^1_0(\Omega)}=\int_\Omega(\nabla \phi\cdot\nabla v-\lambda p \phi v)\dx.
\]
From the assumptions we know that
$F(\lambda_j,\phi_j)=(0,0)$. Moreover, from the simplicity 
assumption \eqref{eqn:hp_simplicity} and  Fredholm Alternative, one
can easily prove the following result
 (see e.g. \cite{Abatangelo2015} for details for a similar operator).

\begin{lemma}
  The functional $F$ is differentiable at $(\lambda_j,\phi_j)$ and its differential
  \begin{align*}
    &\mathrm{d}F(\lambda_j,\phi_j)\colon \R \times H^1_0(\Omega)  \longrightarrow \R\times H^{-1}(\Omega) \\
    &\mathrm{d}F(\lambda_j,\phi_j)(\lambda,\phi)=\left(2\int_\Omega\nabla\phi_j\cdot\nabla\phi\dx,-\Delta\phi-\lambda p\phi_j-\lambda_j p\phi\right)
  \end{align*}
  is invertible.
  \end{lemma}

\begin{lemma}\label{l:7.2}
  Let $\rho\in(0,1/2)$, $K_\rho$ as defined in Proposition \ref{prop:precise_energy_estim_eps} and $R\geq K_\rho$. Then, when $\epsilon\to 0$,
  \begin{equation*}
    v_\jre \longrightarrow \phi_j\qquad\text{in}\quad H^1_0(\Omega),
  \end{equation*}
  where $v_\jre$ is defined in \eqref{eqn:def_vjre}.
\end{lemma}
\begin{proof}
First note that
  \begin{multline*}
    \int_\Omega\abs{\nabla(v_\jre-\phi_j)}^2\dx =\int_{\Omega^\epsilon}\abs{\nabla(\phi_j^\epsilon-\phi_j)}^2\dx\\
    -\int_{\Omega_{R\epsilon}^\epsilon}\abs{\nabla(\phi_j^\epsilon-\phi_j)}^2\dx+\int_{B_{R\epsilon}^+}\abs{\nabla(\vint_\jre-\phi_j)}^2\dx.
  \end{multline*}
  The first term tends to zero because of \eqref{eq:3}. For the second
  and the third term we can exploit the energy estimates in
  Proposition \ref{prop:energy_estim_eps}, Lemma
  \ref{lemma:energy_estim_zero} and Proposition
  \ref{prop:precise_energy_estim_vint} to conclude. 
\end{proof}

\begin{lemma}\label{lemma:vjre_estim}
  Let $\rho\in(0,1/2)$, $K_\rho$ as defined in Proposition \ref{prop:precise_energy_estim_eps} and $R\geq K_\rho$. Then
  \begin{equation*}
    \norm{v_\jre-\phi_j}_{H^1_0(\Omega)}=O\left(\epsilon^{N/2-1}\sqrt{H(\phi_j^\epsilon,K_\rho \epsilon})\right)\qquad\text{as}\quad \epsilon\to 0
  \end{equation*}
  and, in particular,
  \begin{equation}\label{eqn:vjre_estim_2}
    \int_{\Omega\setminus B_{R\epsilon}^+}\abs{\nabla(\phi_j^\epsilon-\phi_j)}^2\dx =O\left(\epsilon^{N-2}H(\phi_j^\epsilon,K_\rho\epsilon)\right)\qquad\text{as}\quad \epsilon\to 0.
  \end{equation}
\end{lemma}
\begin{proof}
Taking into account Lemma \ref{l:7.2} and  \eqref{eq:7}, from the
differentiability of the functional $F$ it follows that
  \[
    F(\lambda_j^\epsilon,v_\jre)=\mathrm{d}F(\lambda_j,\phi_j)(\lambda_j^\epsilon-\lambda_j,v_\jre-\phi_j)+o\left(\abs{\lambda_j^\epsilon-\lambda_j}+\norm{v_\jre-\phi_j}_{H^1_0(\Omega)}\right)
  \]
as $\epsilon\to0$.
  Now let us apply $\mathrm{d}F(\lambda_j,\phi_j)^{-1}$ to both members and obtain
  \begin{multline*}
    \abs{\lambda_j^\epsilon-\lambda_j}+\norm{v_\jre-\phi_j}_{H^1_0(\Omega)}\\\leq
    \norm{\mathrm{d}F(\lambda_j,\phi_j)^{-1}}_{\mathcal L(\R\times H^{-1}(\Omega),
 \R \times H^1_0(\Omega) )}
\norm{F(\lambda_j^\epsilon,v_\jre)}_{\R\times H^{-1}(\Omega)}(1+o(1))
\end{multline*}
  and so
  \begin{equation}\label{eq:8}
    \norm{v_\jre-\phi_j}_{H^1_0(\Omega)}\leq C\left(\abs{\norm{v_\jre}^2_{H^1_0(\Omega)}-\lambda_j}+\norm{-\Delta v_\jre-\lambda_j^\epsilon p v_\jre}_{H^{-1}(\Omega)}\right).
  \end{equation}
  Thanks to \eqref{eqn:estim_vire_1}, Proposition
  \ref{prop:up_low_bound}, and the fact that $f_R(\epsilon)=O(1)$ as
  $\epsilon\to0$ in view of \eqref{eqn:precise_energy_estim_th1} and \eqref{eqn:precise_energy_estim_vint_th1},
  \begin{equation}\label{eq:9}
    \abs{\norm{v_\jre}^2_{H^1_0(\Omega)}-\lambda_j}\leq \abs{\norm{v_\jre}_{H^1_0(\Omega)}-\lambda_j^\epsilon}+\abs{\lambda_j^\epsilon-\lambda_j}=O(\epsilon^{N-2}H(\phi_j^\epsilon,K_\rho\epsilon)).
  \end{equation}
  Let $u\in H^1_0(\Omega)$ be such that $\norm{u}_{H^1_0(\Omega)}\leq 1$. Note that
  \begin{gather*}
    \int_\Omega\nabla v_\jre\cdot\nabla u\dx=\int_{B_{R\epsilon}^+}\nabla\vint_\jre\cdot\nabla u\dx+\int_{\Omega^\epsilon}\nabla\phi_j^\epsilon\cdot\nabla u\dx-\int_{\Omega_{R\epsilon}^\epsilon}\nabla\phi_j^\epsilon\cdot\nabla u\dx\leq \\
    \leq \sqrt{\int_{B_{R\epsilon}^+}\abs{\nabla \vint_\jre}^2\dx} +\lambda_j^\epsilon\int_{\Omega^\epsilon}p\phi_j^\epsilon u\dx +\sqrt{\int_{\Omega_{R\epsilon}^\epsilon}\abs{\nabla\phi_j^\epsilon}^2\dx}.
  \end{gather*}
  So we have that
  \begin{equation}\label{eqn:vjre_estim_1}
  \begin{split}
    &\int_\Omega \nabla v_\jre\cdot\nabla u\dx-\lambda_j^\epsilon\int_\Omega p v_\jre u\dx\leq \\
    &\leq \sqrt{\int_{B_{R\epsilon}^+}\abs{\nabla \vint_\jre}^2\dx} +\lambda_j^\epsilon\left(\int_{\Omega^\epsilon}p\phi_j^\epsilon u\dx-\int_\Omega p v_\jre u\dx\right) +\sqrt{\int_{\Omega_{R\epsilon}^\epsilon}\abs{\nabla\phi_j^\epsilon}^2\dx}.
  \end{split}
  \end{equation}
  Now let us analyze the middle term
  \begin{gather*}
    \int_{\Omega^\epsilon}p\phi_j^\epsilon u\dx-\int_\Omega p v_\jre u\dx=\int_{B_{R\epsilon}^+}p \phi_j^\epsilon u\dx-\int_{B_{R\epsilon}^+} p \vint_\jre u\dx\leq \\
    \leq \const \left(\sqrt{\int_{B_{R\epsilon}^+}\abs{\phi_j^\epsilon}^2\dx} +\sqrt{\int_{B_{R\epsilon}^+}|\vint_\jre|^2\dx}\right)
  \end{gather*}
  where we implicitly used the Poincar\'e Inequality. Thanks to inequality \eqref{eqn:poincare_type} and to the energy estimates made in Proposition \ref{prop:precise_energy_estim_eps}
  \[
    \int_{B_{R\epsilon}^+}\abs{\phi_j^\epsilon}^2\dx\leq \frac{(R\epsilon)^2}{N-1}\int_{B_{R\epsilon}^+}\abs{\nabla\phi_j^\epsilon}^2\dx+\frac{R\epsilon}{N-1}\int_{S_{R\epsilon}^+}|\phi_j^\epsilon|^2\dx =O(\epsilon^N H(\phi_j^\epsilon,K_\rho\epsilon))\quad\text{as }\epsilon\to 0.
  \]
  Then, from \eqref{eqn:vjre_estim_1}, Proposition \ref{prop:energy_estim_eps} and Proposition \ref{prop:precise_energy_estim_vint} we obtain that
  \[
    \int_\Omega \nabla v_\jre\cdot\nabla
    u\dx-\lambda_j^\epsilon\int_\Omega p v_\jre
    u\dx=O\left(\epsilon^{N/2-1}\sqrt{H(\phi_j^\epsilon,K_\rho\epsilon)}\right)\quad\text{as
    }\epsilon\to0
  \]
uniformly with respect to $u\in H^1_0(\Omega)$ with $\norm{u}_{H^1_0(\Omega)}\leq 1$ and hence
  \begin{equation}\label{eq:10}
   \norm{-\Delta v_\jre -\lambda_j^\epsilon p v_\jre}_{H^{-1}(\Omega)}=O\left(\epsilon^{N/2-1}\sqrt{H(\phi_j^\epsilon,K_\rho\epsilon)}\right) \quad\text{as
    }\epsilon\to0.
  \end{equation}
The conclusion follows by combining \eqref{eq:8}, \eqref{eq:9}, and \eqref{eq:10}.
\end{proof}

\begin{corollary}\label{cor:blow_up_bounded}
   Let $\rho\in(0,1/2)$, $K_\rho$ as defined in Proposition \ref{prop:precise_energy_estim_eps} and $R\geq K_\rho$. Then
  \[
    \int_{\frac{1}{\epsilon}\Omega\setminus B_R^+}\abs{\nabla \tilde{\phi}^\epsilon-\epsilon^k H(\phi_j^\epsilon,K_\rho\epsilon)^{-1/2}\nabla W^\epsilon}^2\dx =O(1),\qquad\text{as }\epsilon\to 0,
  \]
  where $\tilde{\phi}^\epsilon$ is defined in \eqref{eqn:def_phi_tilde_Z_R} and $W^\epsilon$ in \eqref{eqn:def_W_U_R}, while $\frac{1}{\epsilon}\Omega=\{\frac{1}{\epsilon}x\colon x\in \Omega\}$.
\end{corollary}
\begin{proof}
  It directly follows from a change of variables in \eqref{eqn:vjre_estim_2}.
\end{proof}

The following Theorem provides a blow-up analysis for scaled eigenfunctions, which contains Theorem \ref{thm:main_2_blowup}.

\begin{theorem}\label{thm:blow_up}
  Let $\rho\in(0,1/2)$ and $K_\rho$ as defined in Proposition
  \ref{prop:precise_energy_estim_eps}. Then 
 \begin{align}
  &\tilde{\phi}^\epsilon\longrightarrow  \frac{1}{\sqrt{\Lambda_\rho}} \Phi \qquad\text{in }\mathcal{H}_R\quad\text{for all }R>2,\label{eqn:blow_up_th1} \\
  &\frac{H(\phi_j^\epsilon,K_\rho\epsilon)}{\epsilon^{2k}}\longrightarrow \Lambda_\rho, \label{eqn:blow_up_th2}\\
  & \frac{\phi_j^\epsilon (\epsilon x)}{\epsilon^k}\longrightarrow \Phi(x)\qquad\text{in }\mathcal{H}_R\quad\text{for all }R>2,
 \end{align}
 as $\epsilon\to 0$, where
 \begin{equation*}
  \Lambda_\rho:=\frac{1}{K_\rho^{N-1}}\int_{S_{K_\rho}^+}\abs{\Phi}^2\ds.
 \end{equation*}

\end{theorem}
\begin{proof}
 Let $\epsilon_n\to 0$. From Corollary \ref{cor:bdd_blow_up} we deduce that, up to a subsequence,
 \[
   \frac{(\epsilon_n)^k}{\sqrt{H(\phi_j^{\epsilon_n},K_\rho\epsilon_n})}\longrightarrow c\geq 0.
 \]
 Since, in view of Proposition \ref{prop:precise_energy_estim_eps}, $\{\tilde{\phi}^{\epsilon_n}\}$ is bounded in $\mathcal{H}_R$, by a diagonal process there exists $\tilde{\Phi}$, with $\tilde{\Phi}\in\mathcal{H}_R$ for all $R>2$, and a subsequence (still denoted by $\epsilon_n$) such that
 \begin{equation}\label{eq:11}
  \tilde{\phi}^{\epsilon_n}\rightharpoonup
  \tilde{\Phi}\qquad\text{weakly in}\ \mathcal{H}_R\quad\text{for all }R>2.
\end{equation}
 Moreover $
 \int_{S_{K_\rho}^+}|\tilde{\phi}^{\epsilon_n}|^2\ds=K_\rho^{N-1}$,
 hence, by compactness of trace embeddings,
 \begin{equation}\label{eqn:blow_up_3}
  \int_{S_{K_\rho}^+}|\tilde{\Phi}|^2\ds=K_\rho^{N-1},
 \end{equation}
thus implying that $\tilde{\Phi}\not\equiv 0$.

Actually we can prove that the convergence in \eqref{eq:11} is
strong. Indeed, consider the equation solved by
$\tilde{\phi}^{\epsilon_n}$:
 \[
  \left\{\begin{aligned}
          -\Delta\tilde{\phi}^{\epsilon_n}
          &=(\epsilon_n)^2\lambda_j^{\epsilon_n}p\,\tilde{\phi}^{\epsilon_n},
          && \text{in
          }\left(\big(-\tfrac1{\epsilon_n},0\big]\times\Sigma\right)\cup B_R^+, \\
          \tilde{\phi}^{\epsilon_n} &=0, &&\text{on }\partial\left(
\left(\big(-\tfrac1{\epsilon_n},0\big]\times\Sigma\right)\cup B_R^+\right)
\setminus S_R^+ ,\\
          \tilde{\phi}^{\epsilon_n}(x) &=\tfrac{\phi_j^{\epsilon_n}(\epsilon_n x)}{\sqrt{H(\phi_j^{\epsilon_n},K_\rho\epsilon_n)}}, &&\text{on }S_R^+.
         \end{aligned}\right.
 \]
 If we consider the restriction to $B_R^+\setminus B_{R/2}^+$ and the
 odd reflection through the hyperplane $x_1=0$, we have that
 $\{\tilde{\phi}^{\epsilon_n}\}$ is bounded in
 $H^2(B_R\setminus B_{R/2})$, where
 $B_R=\{x\in \R^N\colon\abs{x}<R\}$. Hence, up to a subsequence,
 $\frac{\partial\tilde{\phi}^{\epsilon_n}}{\partial \nnu}\to
 \frac{\partial\tilde{\Phi}}{\partial\nnu}$
 in $L^2(S_R^+)$ and therefore
 \[
  \int_{\Pi_R}\abs{\nabla\tilde{\phi}^{\epsilon_n}}^2\dx=(\epsilon_n)^2\lambda_j^{\epsilon_n} \int_{\Pi_R}p\abs{\tilde{\phi}^{\epsilon_n}}^2\dx+\int_{S_R^+}\frac{\partial\tilde{\phi}^{\epsilon_n}}{\partial\nnu}\tilde{\phi}^{\epsilon_n}\ds\to \int_{S_R^+}\frac{\partial\tilde{\Phi}}{\partial\nnu}\tilde{\Phi}\ds=\int_{\Pi_R}|\nabla \tilde{\Phi}|^2\dx. 
 \]
Then we conclude that   $\tilde{\phi}^{\epsilon_n}\to \tilde{\Phi}$
strongly in $\mathcal{H}_R$ for all $R>2$.

From Corollary \ref{cor:blow_up_bounded} it follows that there exist
$c'>0$ and $n_0\in{\mathbb N}$ such that, for all $n\geq n_0$ and
$\tilde{R}>R$,
   \[
    \int_{B_{\tilde{R}}^+\setminus B_R^+}\abs{\nabla
      \tilde{\phi}^{\epsilon_n}-(\epsilon_n)^k
      H(\phi_j^{\epsilon_n},K_\rho{\epsilon_n})^{-1/2}\nabla
      W^{\epsilon_n}}^2\dx \leq c'.
  \]
  Let us recall that $W^{\epsilon_n}\to \psi_k$ in
  $H^1(B_{\tilde{R}}^+)$ and (since the norms are equivalent) also in
  $\mathcal{H}_{\tilde{R}}$: so, passing to the limit as $n\to\infty$
  in the above estimate, we obtain that
  \[
    \int_{B_{\tilde{R}}^+\setminus B_R^+}|\nabla\tilde{\Phi}-c\nabla
    \psi_k|^2\dx\leq c'.
  \]
Since the constant $c'$ is independent on $\tilde{R}$, we deduce that
  \begin{equation}\label{eqn:blow_up_1}
    \int_\Pi|\nabla\tilde{\Phi}-c\nabla\psi_k|^2\dx <+\infty.
  \end{equation}
  Moreover, the function $\tilde{\Phi}$ satisfies the following equation
  \begin{equation}\label{eqn:blow_up_2}
    \left\{\begin{aligned}
             -\Delta\tilde{\Phi} &=0, &&\text{in }\Pi, \\
             \tilde{\Phi}&=0, &&\text{on }\partial\Pi.
           \end{aligned}\right.
  \end{equation}
We claim  that $c>0$. Otherwise, if $c=0$ then, by \eqref{eqn:blow_up_1} and \eqref{eqn:blow_up_2}, we could say that $\tilde{\Phi}=0$, which would contradict \eqref{eqn:blow_up_3}.

From Proposition \ref{prop:def_Phi} we conclude that
$\tilde{\Phi}=c\,\Phi$ and hence, in view of \eqref{eqn:blow_up_3}, 
$c=\Lambda_\rho^{-1/2}$. Since the limit of the sequence
$\{\tilde{\phi}^{\epsilon_n}\}$ is the same for any choice of the
subsequence, we conclude the proof by invoking the \emph{Urysohn's Subsequence
  Principle}.
\end{proof}

\begin{corollary}\label{c:limZ}
For all $R>2$ we have that
 \begin{equation*}
  Z_R^\epsilon\longrightarrow \frac{1}{\sqrt{\Lambda_\rho}}Z_R\qquad\text{in }H^1(B_R^+)\quad\text{as }\epsilon\to 0.
 \end{equation*}
\end{corollary}
\begin{proof}
 From the definitions of the functions $Z_R^\epsilon$ and $Z_R$ (in \eqref{eqn:def_phi_tilde_Z_R} and Lemma \ref{lemma:def_Z_R} respectively)
 \[
 \left\{ \begin{aligned}
     -\Delta(\sqrt{\Lambda_\rho}Z_R^\epsilon-Z_R) & =0, &&\text{in }B_R^+, \\
     \sqrt{\Lambda_\rho}Z_R^\epsilon-Z_R & =0, &&\text{on }\mathcal{C}_R, \\
     \sqrt{\Lambda_\rho}Z_R^\epsilon-Z_R &
     =\sqrt{\Lambda_\rho}\tilde{\phi}^\epsilon-\Phi, &&\text{on
     }S_R^+.
   \end{aligned}\right.
 \]
So $Z_R^\epsilon-Z_R$ is the unique, least energy solution with these
prescribed boundary conditions. Now, let $\eta=\eta_R$ be as defined in \eqref{eqn:def_cut_off}. We have that
\begin{gather*}
  \int_{B_R^+}\abs{\nabla (\sqrt{\Lambda_\rho}
    Z_R^\epsilon-Z_R)}^2\dx\leq
  \int_{B_R^+}\abs{\nabla(\eta(\sqrt{\Lambda_\rho}\tilde{\phi}^\epsilon-\Phi))}^2\dx \leq \\
  \leq 2\int_{B_R^+}\abs{\nabla \eta}^2\abs{\sqrt{\Lambda_\rho}
    \tilde{\phi}^\epsilon-\Phi}^2\dx+2\int_{B_R^+}
  \eta^2\abs{\nabla(\sqrt{\Lambda_\rho}\tilde{\phi}^\epsilon-\Phi)}^2\dx\leq \\
  \leq \frac{32}{R^2}\int_{B_R^+}\abs{\sqrt{\Lambda_\rho}
    \tilde{\phi}^\epsilon-\Phi}^2\dx+2\int_{B_R^+}\abs{\nabla(\sqrt{\Lambda_\rho}\tilde{\phi}^\epsilon-\Phi)}^2\dx
  \to 0
\end{gather*}
as $\epsilon\to 0$, thanks to \eqref{eqn:blow_up_th1} and to the
embedding $\mathcal{H}_R\subset L^2(\Pi_R)$. The conclusion follows
taking into account Poincar\'e Inequality for functions vanishing on a
portion of the boundary.
\end{proof}

\section{Proof of Theorem \ref{thm:main_1}}\label{sec:proof-theor-refthm:m}
 Thanks to Theorem \ref{thm:blow_up} and Corollary \ref{c:limZ}, we know that
  \[
   f_R:=\lim_{\epsilon\to 0}f_R(\epsilon)=\frac{1}{\Lambda_\rho}\int_{B_R^+}\abs{\nabla Z_R}^2\dx-\frac{1}{\Lambda_\rho}\int_{\Pi_R}\abs{\nabla \Phi}^2\dx.
  \]
  Moreover, in view of Proposition \ref{prop:up_low_bound} and \eqref{eqn:blow_up_th2}, we have that, for any $R>\max\{2,K_\rho\}$
  \begin{equation}\label{eqn:liminf_eps}
   C_k(\Sigma) \leq \liminf_{\epsilon\to 0}\frac{\lambda_j-\lambda_j^\epsilon}{\epsilon^{N+2k-2}}\leq \limsup_{\epsilon\to 0}\frac{\lambda_j-\lambda_j^\epsilon}{\epsilon^{N+2k-2}}\leq \Lambda_\rho f_R,
  \end{equation}
    where $C_k(\Sigma)=-2m_k(\Sigma)>0$.
 To complete the proof of our main result it is then enough to show that
   \[
   \lim_{R\to+\infty}\Lambda_\rho f_R= C_k(\Sigma).
  \]
For every  $R>2$ let us define
  \begin{equation}\label{eqn:def_xi_R}
   \xi_R(r):=\int_{S_1^+}Z_R(r\theta)\Psi(\theta)\dth \qquad\text{for }0\leq r\leq R.
  \end{equation}

  \begin{lemma}\label{lemma:f_R_a}
  There holds
   \begin{gather}
    \int_{S_R^+}\frac{\partial(Z_R-\psi_k)}{\partial \nnu}(\Phi-\psi_k)\ds\longrightarrow 0 \qquad\text{as }R\to +\infty, \label{eqn:f_R_lemma_a_th1}\\
    \int_{S_R^+}\frac{\partial(\psi_k-\Phi)}{\partial \nnu}(\Phi-\psi_k)\ds\longrightarrow 0 \qquad\text{as }R\to +\infty. \label{eqn:f_R_lemma_a_th2}
   \end{gather}
  \end{lemma}
  \begin{proof}
   In order to prove \eqref{eqn:f_R_lemma_a_th1}, we first take into account the equation solved by $Z_R-\psi_k$, i.e.
   \begin{equation}\label{eqn:f_R_lemma_a1}
    \left\{\begin{aligned}
     -\Delta(Z_R-\psi_k) &=0,& & \text{in }B_R^+, \\
     Z_R-\psi_k &=0 ,&& \text{on }\mathcal{C}_R ,\\
     Z_R-\psi_k &=\Phi-\psi_k ,&& \text{on }S_R^+.
    \end{aligned}\right.
   \end{equation}
  Let $\eta=\eta_R$ as defined in \eqref{eqn:def_cut_off}. Testing
  \eqref{eqn:f_R_lemma_a1} with $\eta(\Phi-\psi_k)$, we obtain that
   \begin{equation*}
    \int_{B_R^+}\nabla(Z_R-\psi_k)\cdot \nabla (\eta(\Phi-\psi_k))\dx=\int_{S_R^+}\frac{\partial(Z_R-\psi_k)}{\partial \nnu}(\Phi-\psi_k)\ds.
   \end{equation*}
   Then, by the Dirichlet principle,
   \begin{gather*}
    \int_{S_R^+}\frac{\partial(Z_R-\psi_k)}{\partial \nnu}(\Phi-\psi_k)\ds \leq \sqrt{\int_{B_R^+}\abs{\nabla(Z_R-\psi_k)}^2\dx}\sqrt{\int_{B_R^+}\abs{\nabla(\eta(\Phi-\psi_k))}^2\dx}\leq \\
     \leq \int_{B_R^+}\abs{\nabla(\eta(\Phi-\psi_k))}^2\dx\leq \frac{32}{R^2}\int_{B_R^+\setminus B_{R/2}^+}\abs{\Phi-\psi_k}^2\dx+2\int_{B_R^+\setminus B_{R/2}^+}\abs{\nabla(\Phi-\psi_k)}^2\dx\to 0    
   \end{gather*}
   as $R\to+\infty$, thanks to the fact that $\Phi-\psi_k\in \mathcal{D}^{1,2}(\Pi)$ and to Hardy's inequality (reasoning as in Lemma \ref{lemma:conv_U_R}).
   
   For the second part, since $-\Delta(\Phi-\psi_k)=0$ in $\Pi\setminus \Pi_R$ and $\Phi-\psi_k=0$ on $\{x_1=0\}\setminus\Sigma$, then
   \[
    \int_{S_R^+}\frac{\partial(\psi_k-\Phi)}{\partial \nnu}(\Phi-\psi_k)\ds=\int_{\Pi\setminus\Pi_R}\abs{\nabla (\Phi-\psi_k)}^2\dx\to 0
   \]
   as $R\to+\infty$. 
  \end{proof}

\begin{lemma}\label{lemma:f_R}
 We have that $ \lim_{R\to+\infty}\Lambda_\rho f_R=-2m_k(\Sigma)$.
\end{lemma}
\begin{proof}
 Thanks to Lemma \ref{lemma:f_R_a} we know that
 \begin{equation}\label{eq:12}
  \lim_{R\to+\infty}\Lambda_\rho f_R=\lim_{R\to +\infty}\int_{S_R^+}\left( \frac{\partial Z_R}{\partial \nnu}-\frac{\partial \Phi}{\partial \nnu} \right)\psi_k\ds.
\end{equation}
 From the definition of $\zeta$ \eqref{eqn:def_zeta_r} and from \eqref{eqn:zeta'_r} we deduce that
\begin{equation}\label{eq:13}
  \int_{S_R^+}\frac{\partial \Phi}{\partial \nnu} \psi_k\ds=R^{N+k-1}\zeta'(R)=k\gamma_N R^{N+2k-2}+(N+k-2)(\gamma_N-\zeta(1)).
\end{equation}
 It's easy to verify that the function $\xi_R$ defined in
 \eqref{eqn:def_xi_R} satisfies the following ODE
 \[
  \left( r^{N+2k-1}(r^{-k}\xi_R(r))' \right)'=0\qquad\text{in }(0,R).
 \]
 By integration, we obtain
 \[
  r^{N+k-2}\xi_R(r)=r^{N+2k-2}R^{-k}\xi_R(R)-\frac{C}{N+2k-2}+\frac{C}{N+2k-2}r^{N+2k-2}R^{-N-2k+2}.
 \]
 Since $Z_R$ is regular at $0$, we have necessarily that $C=0$; hence
 \begin{equation*}
  \xi_R(r)=\left( \frac{r}{R} \right)^k \xi_R(R).
 \end{equation*}
 From the definition of $\xi_R$ \eqref{eqn:def_xi_R} we have
 \begin{equation}\label{eqn:f_R_1}
  \int_{S_R^+}\frac{\partial Z_R}{\partial \nnu}\psi_k\ds=R^{N+k-1}\xi_R'(R)=k R^{N+k-2}\xi_R(R)=k R^{N+k-2} \zeta(R).
 \end{equation}
 Then, from \eqref{eq:12}, \eqref{eq:13}, \eqref{eqn:f_R_1}, \eqref{eqn:zeta_r_7} and \eqref{eqn:th_lemma_zeta_1}
 \begin{align*}
  \lim_{R\to +\infty} \Lambda_\rho f_R&=\lim_{R\to +\infty} \left(
    kR^{N+k-2}\zeta(R)-
k\gamma_N R^{N+2k-2}-(N+k-2)(\gamma_N-\zeta(1))
\right) \\
&  = \lim_{R\to +\infty} R^{N+k-2}(N+2k-2)(\zeta(R)-\gamma_N R^k)\\
&=  (N+2k-2)(\zeta(1)-\gamma_N) = -2m_k(\Sigma),
 \end{align*}
 thus concluding the proof.
\end{proof}

We are now able to prove our main result.
\begin{proof}[Proof of Theorem \ref{thm:main_1}]
 From \eqref{eqn:liminf_eps} and Lemma \ref{lemma:f_R} we conclude that
  \[
   \liminf_{\epsilon\to 0}\frac{\lambda_j-\lambda_j^\epsilon}{\epsilon^{N+2k-2}}=\limsup_{\epsilon\to 0}\frac{\lambda_j-\lambda_j^\epsilon}{\epsilon^{N+2k-2}}=C_k(\Sigma),
  \]
  thus completing the proof. 
\end{proof}

\section{Appendix}

 It is well known that the classical Hardy's Inequality
 \[
  \left(\frac{N-2}{2}\right)^2\int_{\R^N}\frac{\abs{u}^2}{\abs{x}^2}\dx\leq
  \int_{\R^N}\abs{\nabla u}^2\dx,\quad u\in C^{\infty}_c(\R^N),\quad N\geq 3,
 \]
 fails in dimension 2. However we observe, in the following theorem,
 that, under a vanishing condition on part of the domain (at least on a
 half-line), it is possible to recover a Hardy-type Inequality even in
 dimension 2.

 Let $\mathbf{p}=(x_\mathbf{p},0)\in \R^2$ with $x_\mathbf{p}>0$ and
 let $s_\mathbf{p}:=\{(x,0)\colon x\geq x_\mathbf{p}\}$. Let
 $\mathcal{D}_\mathbf{p}$ denote the completion of the space
 $C_c^\infty(\R^2\setminus s_{\mathbf{p}})$ with respect to the norm
  \[
    \norm{u}_{\mathcal{D}_\mathbf{p}}:=\left( \int_{\R^2}\abs{\nabla u}^2\dx \right)^{1/2}.
  \]
  Let us consider  the function
  \[
    \theta_\mathbf{p}\colon \R^2\setminus s_\mathbf{p}\longrightarrow
    (0,2\pi),
\quad 
    \theta_\mathbf{p}(x_\mathbf{p}+r\cos t, r\sin t)= t.
  \]
  We have that $\theta_\mathbf{p}\in C^\infty(\R^2\setminus s_\mathbf{p})$.
  \begin{theorem}\label{thm:2dim_hardy}
    For all $\phi\in C_c^\infty(\R^2\setminus s_{\mathbf{p}})$ 
\begin{equation}\label{eq:14}
      \frac{1}{4}\int_{\R^2}\frac{\abs{\phi(z)}^2}{\abs{z-\mathbf{p}}^2}\dz\leq \int_{\R^2}\abs{\nabla\phi(z)}^2\dz.
    \end{equation}
Moreover the space $\mathcal{D}_\mathbf{p}$ can be  characterized as
\[
\mathcal{D}_\mathbf{p}=\left\{u\in L^1_{\rm loc}(\R^2):\nabla u\in
L^2(\R^2),\ \tfrac{u}{\abs{z-\mathbf{p}}}\in L^2(\R^2),\text{ and
}u=0\text{ on }s_{\mathbf{p}}\right\}
\]
and inequality \eqref{eq:14} holds for every $\varphi\in \mathcal{D}_\mathbf{p}$.

  \end{theorem}
  \begin{proof}
    Let $\phi\in C_c^\infty(\R^2\setminus s_\mathbf{p})$ and let
    $\tilde{\phi}(z):=\phi(z) e^{i\frac{\theta_\mathbf{p}(z)}{2}}\in
    C_c^\infty(\R^2\setminus s_\mathbf{p},{\mathbb C})$.
    By direct calculations, we have that
   \[
     i\nabla \phi(z)=e^{-i\frac{\theta_p(z)}{2}}(i\nabla +\mathrm{A}_{\mathbf{p}})\tilde{\phi}(z),
   \]
   where
   \[
    \mathrm{A}_{\mathbf{p}}(x,y):=\frac{1}{2}\left( \frac{-y}{(x-x_\mathbf{p})^2+y^2},\frac{x-x_\mathbf{p}}{(x-x_\mathbf{p})^2+y^2} \right)
   \]
is the Aharonov-Bohm vector potential with pole $\mathbf{p}$ and
circulation $1/2$. 
   Now let us compute the $L^2$-norm of $\abs{i\nabla \phi}$ and use the Hardy's Inequality for Aharonov-Bohm operators (see \cite{Laptev1999}):
   \[
     \int_{\R^2}\abs{\nabla\phi}^2\dz =\int_{\R^2}\abs{(i\nabla +\mathrm{A}_{\mathbf{p}})\tilde{\phi}}^2\dz\geq \frac{1}{4}\int_{\R^2}\frac{\abs{\tilde{\phi}}^2}{\abs{z-\mathbf{p}}^2}\dz =\frac{1}{4}\int_{\R^2}\frac{\abs{\phi}^2}{\abs{z-\mathbf{p}}^2}\dz.
   \]
The second part of the statement follows from \eqref{eq:14} by classical completion and  density arguments.  
   \end{proof}
  \begin{corollary}\label{cor:2dim_hardy}
    There exists $C=C(\mathbf{p})>0$ such that
    \begin{equation}\label{eqn:hardy2}
      \int_{\R^2}\frac{\abs{\phi(z)}^2}{1+\abs{z}^2}\dz\leq C\int_{\R^2}\abs{\nabla\phi(z)}^2\dz\quad\text{for all } \phi\in\mathcal{D}_\mathbf{p}.
    \end{equation}
  \end{corollary}
  \begin{proof}
    We observe that there exists $K=K(\abs{\mathbf{p}})>0$ such that
    \[
    \abs{z-\mathbf{p}}^2\leq K(\abs{\mathbf{p}})(1+\abs{z}^2)
    \quad\text{for all }z\in \R^2.
    \]
    Therefore the claim easily follows from Theorem \ref{thm:2dim_hardy} with $C(\mathbf{p})=4 K(\abs{\mathbf{p}})$.
  \end{proof}
  
  We conclude this appendix by recalling from \cite{Abatangelo2015}
  the following lemma about maxima of quadratic forms depending on a
  parameter, which we used in Section \ref{sec:estim-diff-texorpdfs}.
  
  \begin{lemma}\label{lemma:quadratic_form}
   For every $\epsilon>0$ let us consider a quadratic form
   \[
   \begin{aligned}
   &Q_\epsilon\colon \R^j \longrightarrow \R, \\
    &Q_\epsilon(\xi_1,\dots,\xi_j)=\sum_{i,n=1}^j M_{i,n}(\epsilon)\xi_i \xi_n,
    \end{aligned}
   \]
   with real coefficients $M_{i,n}(\epsilon)$ such that $M_{i,n}(\epsilon)=M_{n,i}(\epsilon)$. Let us assume that there exist $\alpha>0$, $\epsilon\mapsto \sigma(\epsilon)\in\R$ with $\sigma(\epsilon)\geq 0$ and $\sigma(\epsilon)=O(\epsilon^{2\alpha})$ as $\epsilon\to 0$, and $\epsilon\mapsto \mu(\epsilon)\in\R$ with $\mu(\epsilon)=O(1)$ as $\epsilon\to 0$, such that the coefficients $M_{i,n}(\epsilon)$ satisfy the following conditions:
   \begin{align*}
    & M_{j,j}(\epsilon)=\sigma(\epsilon)\mu(\epsilon), \\
    & \text{for all }i<j~M_{i,i}(\epsilon)\to M_i<0,~\text{as }\epsilon\to 0, \\
    & \text{for all }i<j~M_{i,j}(\epsilon)=O(\epsilon^{\alpha}\sqrt{\sigma(\epsilon)})~\text{as }\epsilon\to 0, \\
    & \text{for all }i,n<j~\text{with }i\neq n~M_{i,n}=O(\epsilon^{2\alpha})~\text{as }\epsilon\to 0, \\
    & \text{there exists }M\in \mathbb{N}~\text{such that }\epsilon^{(2+M)\alpha}=o(\sigma(\epsilon))~\text{as }\epsilon\to 0.
   \end{align*}
   Then
   \[
    \max_{\substack{\xi \in\R^j \\ \norm{\xi}=1}}Q_\epsilon(\xi)=\sigma(\epsilon)(\mu(\epsilon)+o(1))\quad\text{as }\epsilon\to 0,
   \]
   where $\norm{\xi}=\norm{(\xi_1,\dots,\xi_j)}=\big( \sum_{i=1}^j\xi_i^2 \big)^{1/2}$.   
  \end{lemma}

 \paragraph{Acknowledgements} The authors are partially supported by the INDAM-GNAMPA 2018 grant ``Formula di monotonia e applicazioni: problemi frazionari e stabilit\`a spettrale rispetto a perturbazioni del dominio''.
 V. Felli is partially supported by the PRIN 2015
grant ``Variational methods, with applications to problems in
mathematical physics and geometry''.

\bibliographystyle{siam}

\end{document}